\documentclass[aip,cha,reprint,longbibliography]{revtex4-1}
\usepackage{mathtools}
\usepackage{amsmath}
\usepackage{amssymb,amsthm}
\usepackage[pdftex]{graphicx}
\usepackage[unicode=true,pdfusetitle,
 bookmarks=true,bookmarksnumbered=false,bookmarksopen=false,
 breaklinks=false,pdfborder={0 0 1},backref=false,colorlinks=false]
 {hyperref}
\usepackage{prettyref}

\makeatletter
  \theoremstyle{remark}
  \newtheorem{rem}{\protect\remarkname}
  \theoremstyle{definition}
  \newtheorem{defn}{\protect\definitionname}
  \theoremstyle{plain}
  \newtheorem{lem}{\protect\lemmaname}
\theoremstyle{plain}
\newtheorem{thm}{\protect\theoremname}
  \theoremstyle{plain}
  \newtheorem{cor}{\protect\corollaryname}
 \theoremstyle{definition}
  \newtheorem{example}{\protect\examplename}

\newrefformat{prop}{Proposition \ref{#1}}
\newrefformat{lem}{Lemma \ref{#1}}
\newrefformat{thm}{Theorem \ref{#1}}
\newrefformat{cor}{Corollary \ref{#1}}
\newrefformat{rem}{Remark \ref{#1}}
\newrefformat{example}{Example \ref{#1}}
\newrefformat{def}{Definition \ref{#1}}
\newrefformat{eq}{(\ref{#1})}

\newcommand{\R}{\mathbb{R}}

\makeatother

\providecommand{\definitionname}{Definition}
\providecommand{\examplename}{Example}
\providecommand{\lemmaname}{Lemma}
\providecommand{\remarkname}{Remark}
\providecommand{\corollaryname}{Corollary}
\providecommand{\theoremname}{Theorem}

\begin{document}

\title{Attracting Lagrangian Coherent Structures on Riemannian manifolds}
\date{\today}
\author{Daniel Karrasch}
\affiliation{ETH Z\"urich, Institute for Mechanical Systems, Leonhardstrasse 21, 8092 Z\"urich, Switzerland}
\email{karrasch@imes.mavt.ethz.ch}

\begin{abstract}
It is a wide-spread convention to identify repelling Lagrangian Coherent Structures (LCSs)
with ridges of the forward finite-time Lyapunov exponent (FTLE) field, and attracting
LCSs with ridges of the backward FTLE. We show that in two-dimensional incompressible
flows attracting LCSs appear as ridges of the forward FTLE field, too. This raises the issue of characterization of
attracting LCSs from a forward finite-time Lyapunov analysis. To this end,
we extend recent results by Haller \& Sapsis (2011) \onlinecite{Haller2011b}, regarding the
relation between forward and backward maximal and minimal FTLEs, to both
the whole finite-time Lyapunov spectrum and to stretch directions. This
is accomplished by considering the singular value decomposition (SVD)
of the linearized flow map.
By virtue of geometrical insights from the SVD, we give a short and direct
proof of the main result of Farazmand \& Haller (2013) \onlinecite{Farazmand2013},
and prove a new characterization of attracting LCSs in forward time for
Haller's variational approach to hyperbolic LCSs \onlinecite{Haller2011}.
We apply these results to the attracting FTLE ridge of the incompressible saddle flow.
\end{abstract}

\maketitle

\begin{quotation}
Hyperbolic Lagrangian Coherent Structures (LCSs) are material surfaces that
act as cores of mixing patterns in complex, unsteady and finite-time dynamical 
systems through material repulsion or attraction. In practice, repelling LCSs 
are often identified with ridges of the forward finite-time Lyapunov exponent 
(FTLE) field. Attracting LCSs are then defined as repelling LCSs in backward 
time. It is known that material structures experiencing strong shear induce large 
particle separation, and can hence appear as FTLE ridges as well. In the most frequently 
considered case of two-dimensional incompressible flows, we point out that 
also attracting LCSs are characterized by strong tangential particle separation, 
and may hence appear as forward FTLE ridges. We prove characterizations of 
attracting LCSs in forward time, which helps to determine the dynamical 
cause of particle separation.
\end{quotation}

\section{Introduction}

In recent years, there has been an increasing interest in dynamical
systems given on a finite time interval, driven by applications in
geophysical fluid flows, biological models and engineering. Hyperbolic
Lagrangian Coherent Structures (LCSs), i.e., co\-di\-men\-sion-one material
surfaces with locally the strongest normal repulsion or attraction \cite{Haller2000},
have been identified as the key structures governing transient pattern formation,
transport and mixing\cite{Peacock2013}.
As such, they are considered as finite-time analogues to stable and
unstable manifolds of hyperbolic equilibria/trajectories in steady/unsteady
flows admitting time-asymptotic solutions. Many approaches to numerical
LCS detection have been developed, and most of them identify
hyperbolic LCSs with ridges in scalar separation measure fields
such as finite-time Lyapunov exponents (FTLEs)\cite{Haller2001,Shadden2005,Lekien2007}, finite-size Lyapunov
exponents (FSLEs)\cite{Joseph2002}, relative dispersion\cite{Bowman1999}, and finite-time entropy\cite{Froyland2012}.

A relation between hyperbolic LCSs and the FTLE field was first suggested
by Haller \cite{Haller2001}. The guiding in\-tu\-i\-tion \cite{Haller2001,Haller2001a}
was that repelling LCSs should be indicated by curves (or surfaces)
of (locally) maximal values---subsumed by an intuitive notion of \emph{ridges}---of
some separation measure, computed in forward time. In contrast,
attracting LCSs were thought to be indicated by ridges obtained from
a backward-time computation. This intuition has been adopted
for the FTLE \cite{Shadden2005,Lekien2007}, and is also present
in the majority of approaches based on ridges of separation measures.

In this article, we show 
that in the two-dimensional incompressible case,
strongly \emph{attracting} structures necessarily induce
strong particle separation and are hence indicated by forward FTLE ridges.
This is in contradiction to the wide-spread
hypothesis that any forward FTLE ridge corresponds to a repelling flow structure.
There are several conceptual and computational issues related to that hypothesis \cite{Haller2002,Branicki2010,Haller2011}, including an example of an FTLE ridge
induced by strong shear.
To the best of our knowledge, the phenomenon of normally attracting forward FTLE ridges
has been unknown and adds another aspect of mathematical inconsistency to purely ridge-based LCS approaches. After all, FTLE ridges may be induced
by all sorts of dynamical phenomena, ranging from normal attraction
via shear to normal repulsion.

Our example raises the question of how to determine attracting LCSs in a 
forward finite-time Lyapunov analysis. This question has been first considered 
in the context of the geodesic approach to hyperbolic LCSs\cite{Farazmand2013}.
As we show in this work, one possible solution is to include principal directions
in the LCS approach, as provided by the variational theory \cite{Haller2011,Farazmand2012a,Karrasch2012} and the more recent geodesic theory of LCSs \cite{Haller2012,Farazmand2014a,Karrasch2013c}.

In \prettyref{sec:Geometry}, we extend recent results on the relation between 
forward and backward FTLEs\cite{Haller2011b} to both the whole finite-time Lyapunov
spectrum and to associated principal directions. We develop the theory in the 
setting of Riemannian manifolds, which was first considered in Ref.~\onlinecite{Lekien2010} 
and is of interest in large-scale geophysical fluid flow applications. Our analysis
strongly benefits from considering the singular value decomposition (SVD) of 
the deformation gradient tensor. 
Remarkably, even though the SVD is a well-established tool for the computation of
Lyapunov spectra and Lyapunov vectors of time-asymptotic dynamical systems \cite{Greene1987},
its use in the finite-time context has been limited mostly to an alternative computation 
tool of the FTLE.
Our analysis, however, also includes fields of subdominant singular values and singular vectors.
This turns out to be of significant theoretical and computational advantage\cite{Karrasch2013c}.

\prettyref{sec:LCS} is devoted to the characterization of attracting LCSs in
the geodesic \cite{Farazmand2013,Blazevski2014,Farazmand2014a} and the variational
\cite{Haller2011} sense in forward time. In the first case, we provide a short and
direct geometric proof of the main result of Ref.~\onlinecite{Farazmand2013}. 
In the latter, the characterization is new and reveals that in two-dimensional 
incompressible flows, variational attracting LCSs
satisfy ridge-type conditions for the forward FTLE field. We apply both approaches to the
saddle flow discussed in \prettyref{sec:saddle} and show that both methods
detect repelling and attracting LCSs correctly, irrespective of the orientation
of a visual FTLE ridge.

\section{Example: A Nonlinear Incompressible Saddle}\label{sec:saddle}

Consider dynamics around the autonomous, nonlinear, incompressible saddle, described
by the Hamiltonian\cite{Karrasch2013a}
\[
H(x,y)=-L\tanh(q_{1}x)\tanh(q_{2}y).
\]
Here, $L>0$ governs the strength of hyperbolicity, and $q_{1},q_{2}$
localize the saddle behavior. Trajectories are given as solutions of the 
ordinary differential equation
\[
\begin{split}\dot{x} & =\partial_{y}H(x,y)=-Lq_{2}\left(1-\tanh(q_{2}y){}^{2}\right)\tanh\left(q_{1}x\right),\\
\dot{y} & =-\partial_{x}H(x,y)=Lq_{1}\tanh\left(q_{2}y\right)\left(1-\tanh(q_{1}x)^{2}\right).
\end{split}
\]
We set the parameters to $L=2$, $q_{1}=1$ and $q_{2}=0.15$, and
the resulting vector field on $\left[-1,1\right]^{2}$ is shown in
Fig.\ \ref{fig:nlin_saddle}(a). Here, the origin is a hyperbolic steady state, with $x$- and $y$-axis as the classic stable and unstable manifolds, respectively. For any integration time, segments of the $x$-axis centered around the origin are normally repelling, whereas similar segments of the $y$-axis are normally attracting.

\begin{figure}
\centering
{(a)\quad \includegraphics[width=0.6\columnwidth]{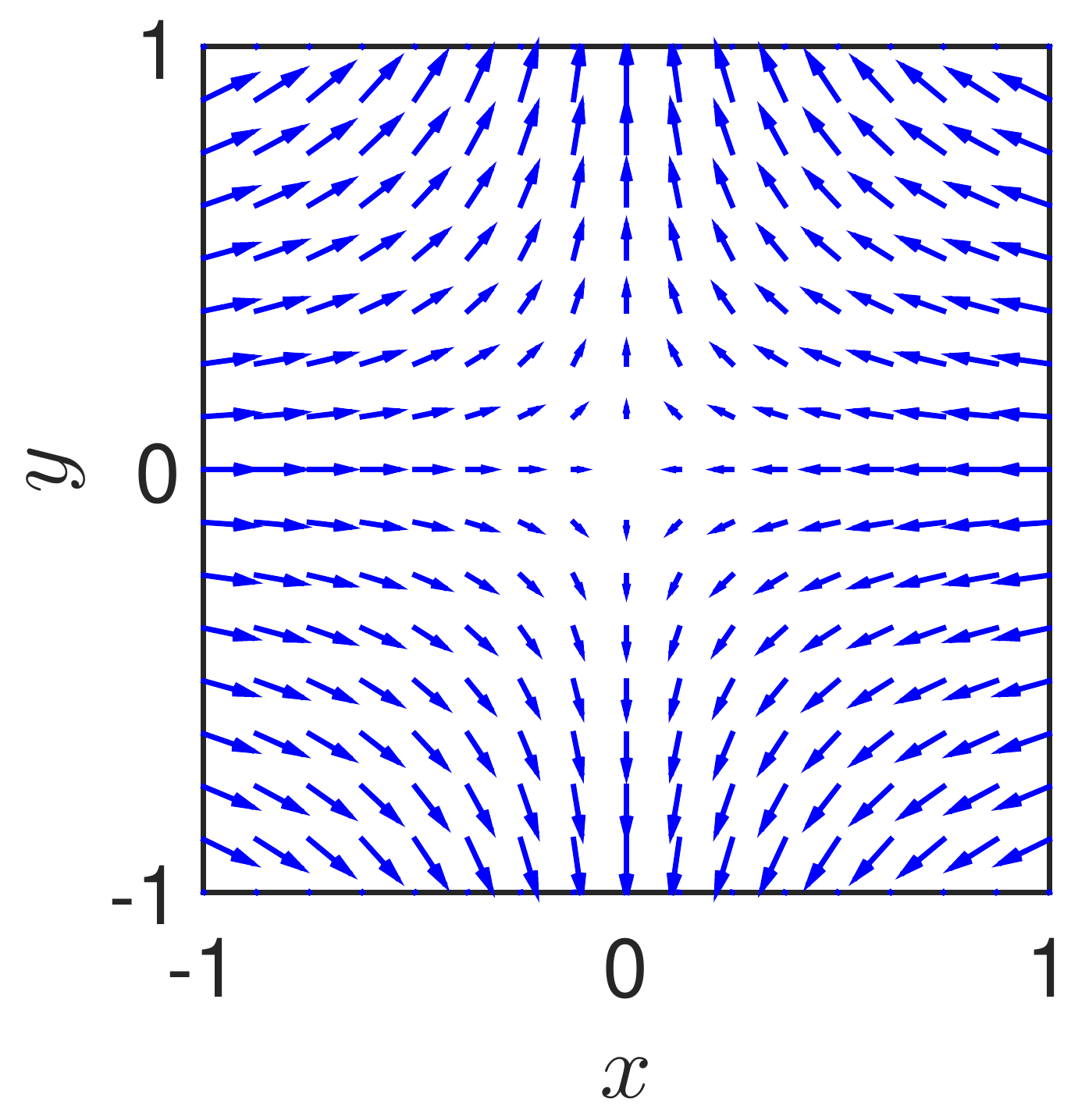}}\\
{(b)\quad \includegraphics[width=0.6\columnwidth]{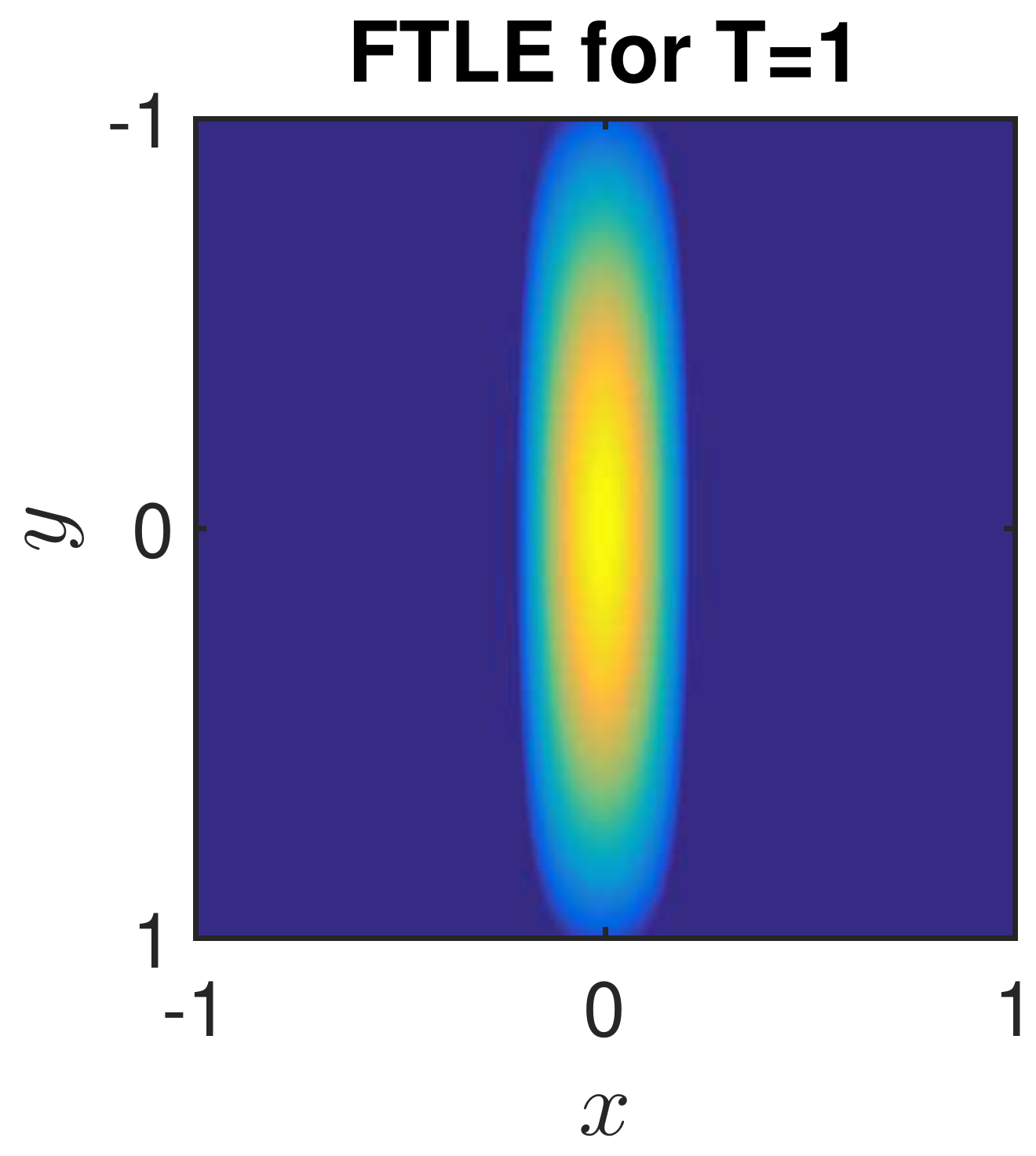}}\\
{(c)\quad \includegraphics[width=0.6\columnwidth]{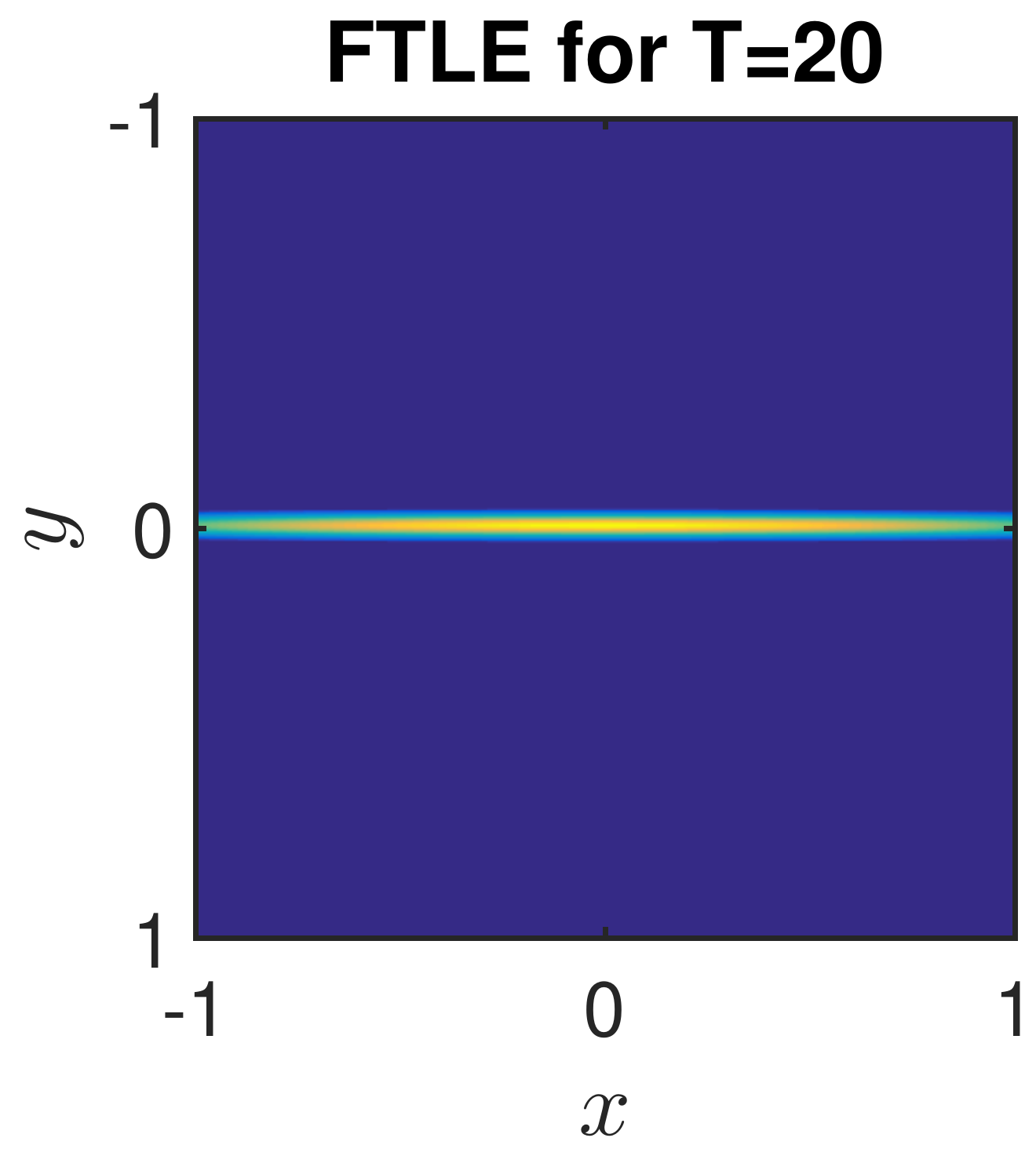}}
\caption{Nonlinear incompressible saddle. (a) Vector field. (b)-(c) FTLE fields for integration time $T=1$ (b) and $T=20$ (c). The FTLE fields are cut off from below at levels of $0.29$ (b) and $0.25$ (c) for visualization purposes.}
\label{fig:nlin_saddle}
\end{figure}

For (forward) integration time $T=1$, we observe that the
$y$-axis appears visually as an FTLE ridge (Fig.\ \ref{fig:nlin_saddle}(b))
and would be commonly misinterpreted as a repelling LCS on the time 
interval $[0,1]$. This is in contradiction to its attracting nature. 
For a longer integration time of $T=20$, the $x$-axis appears visually 
as an FTLE ridge, while the $y$-axis ridge seems to have disappeared, see Fig.\ \ref{fig:nlin_saddle}(c).

For both integration times, however, segments of the respective complimentary axes
can be shown to be FTLE \emph{height ridges} \cite{Eberly1994}.
First, the directional derivatives of
the FTLE field along the $x$- and $y$-axis in the respective normal
direction vanish identically over the whole considered length. Second,
the second-order directional derivatives are negative on $x\in\left[-0.25,0.25\right]$
and $y\in\left[-0.14,0.14\right]$, respectively, see Fig.\ \ref{fig:nlin_saddle_deriv}.
Thus, locally around the origin, repelling and attracting material
lines would be detected by a height ridge extraction algorithm, yielding a \emph{cross of ridges}
around the hyperbolic saddle. We will show later that segments of the normally atttracting
$y$-axis persist as an FTLE ridge for \emph{any} integration time. As we demonstrated, 
for increasing integration time they shrink in length, eventually below grid resolution.

\begin{figure}
\centering
{(a)\quad \includegraphics[width=0.6\columnwidth]{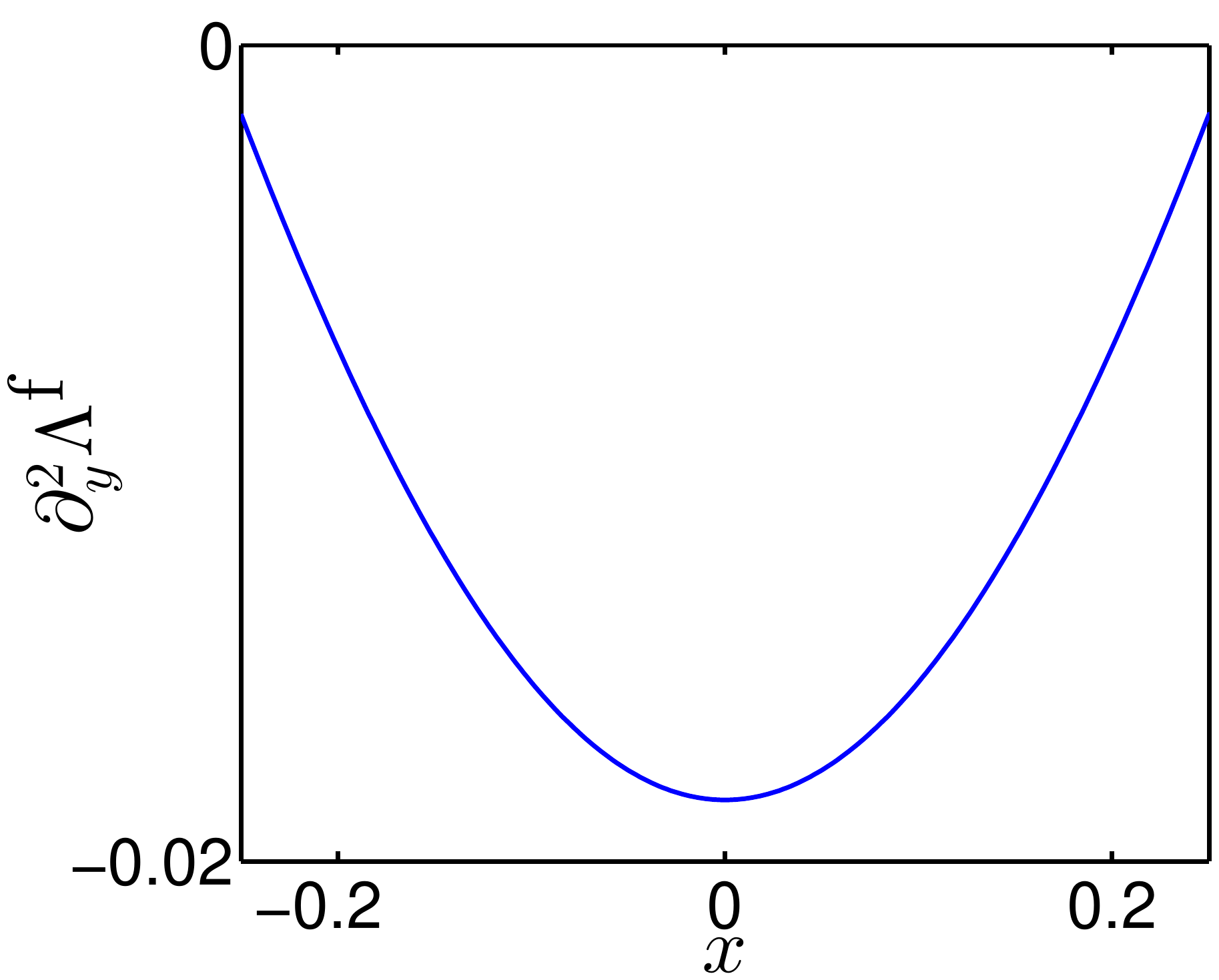}}\\
{(b)\quad \includegraphics[width=0.6\columnwidth]{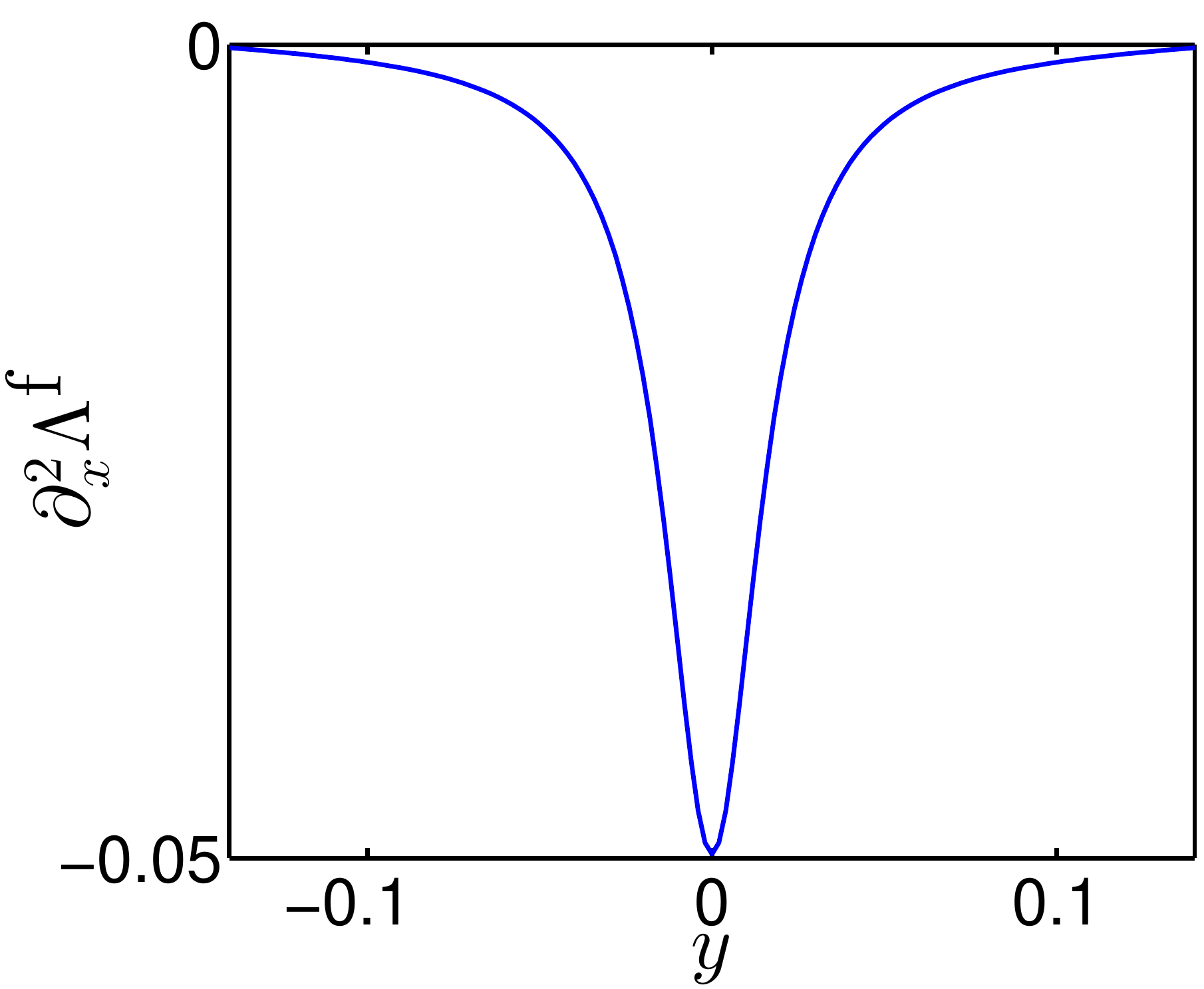}}
\caption{(a) Plot of the second-order derivative w.r.t.\ $y$ of $\Lambda^{\textnormal{f}}$
along the $y=0$ axis for integration time $T=1$. (b) Plot of
the second-order derivative w.r.t.\ $x$ of $\Lambda^{\textnormal{f}}$
along the $x=0$ axis for integration time $T=20$. Both plots show that
also the respective axes which do not show up visually in the FTLE field,
are in fact FTLE (height) ridges, since the first-order derivatives vanish identically.}
\label{fig:nlin_saddle_deriv}
\end{figure}

Based on this example, FTLE fields, which have been computed
over a sufficiently long integration time, seem to indicate
repelling or shearing material structures rather than attracting ones.
In a flow with a priori unknown time scales or simply in a limited data
set, however, it may be unclear whether the chosen/available integration
time suffices to make the appearance of attracting LCSs as FTLE ridges
unlikely. This in turn necessitates additional post-processing.

As indicated by this example and later theoretical results, defining
attracting and repelling LCSs as FTLE
ridges in backward and forward time, respectively, is mathematically
inconsistent. This inconsistency does not restrict to the FTLE field
as one representative of a separation measure.
The reason is that separation
measures do not indicate \emph{directions} of stretching,
but merely the pure fact of stretching. The orientation
of a specified structure with respect to certain directions
is defined only after its determination, and obviously cannot
be obtained from the scalar field a priori. Additionally, in incompressible
two-dimensional vector fields repulsion and attraction balance at
each point, such that attracting structures come along with
strong particle separation and may appear as ridges.

As we show in this work, one way to resolve this observed inconsistency
is to consider directional information
in LCS theory and computations.

\section{The Geometry of Linearized Deformation\label{sec:Geometry}}

In this section, we study finite-time flows on Riemannian manifolds.
Thereby, we adopt the manifold framework of Ref.~\onlinecite{Lekien2010}, where
aspects related to the computation of the scalar FTLE field on
non-Euclidean manifolds are discussed. We keep our presentation 
coordinate-free, and recall coordinate-representation issues in Appendix \ref{sec:representation}.

Let $\mathcal{M}$ be an $n$-dimensional Riemannian manifold, i.e.\ a smooth
manifold with a (Riemannian) metric. We denote the tangent space of $\mathcal{M}$
at $\mathbf{x}\in\mathcal{M}$ by $T_{\mathbf{x}}\mathcal{M}$. The
metric gives rise to an inner product, and hence a norm, on each tangent space, and therefore
induces a notion of angles between and length of tangent vectors.

Consider a nonautonomous $C^1$ velocity field $\mathbf{u}$ on $\mathcal{M}$,
defined on a finite time interval $I\coloneqq\left[t_{1},t_{2}\right]$.
The associated (particle) motion is then given by solutions of the ordinary differential equation
\begin{equation}
\dot{\mathbf{x}}=\mathbf{u}(t,\mathbf{x}).\label{eq:ODE}
\end{equation}
The $t_{1}$-based flow map is denoted by
$\mathbf{F}_{t_{1}}^{t}\colon\mathcal{D}\subseteq\mathcal{M}\to\mathbf{F}_{t_{1}}^{t}[\mathcal{D}]\subseteq\mathcal{M}$
and maps initial values $\mathbf{x}_{1}\in\mathcal{D}$ from time $t_{1}$ to their
position at time $t\in I$ according to the unique solution of \eqref{eq:ODE}
passing through $\mathbf{x}_{1}$ at time $t_{1}$. Recall that the
flow map is continuously differentiable in $\mathbf{x}_{1}$.
In continuum mechanical terms, the flow map can be
thought of as a one-parametric family of \emph{deformations} of the flow domain $\mathcal{D}$ \cite{Truesdell2004}.

Next, we consider the linearization of the flow map, also referred to as the
\emph{deformation gradient (tensor)} in continuum mechanics,
\begin{align*}
D\mathbf{F}_{t_1}^t=\left(\mathbf{F}_{t_{1}}^{t}\right)_*&\colon T_{\mathcal{D}}\mathcal{M}\to T_{\mathbf{F_{t_1}^t}[\mathcal{D}]}\mathcal{M},
\intertext{or pointwise}
D\mathbf{F}_{t_{1}}^{t}(\mathbf{x}_{1})=\left(\mathbf{F}_{t_{1}}^{t}\right)_*(\mathbf{x}_{1})&\colon
T_{\mathbf{x}_1}\mathcal{M}\to T_{\mathbf{F}_{t_1}^t(\mathbf{x}_1)}\mathcal{M},
\end{align*}
for $\mathbf{x}_1\in\mathcal{D}$, where the sub-index $*$ denotes push-forward \cite{Abraham1988}. Hence, the deformation-gradient tensor field is a vector bundle isomorphism over $\mathbf{F}$ \cite{Abraham1988}, see Figs.\ \ref{fig:flow_tangent} and \ref{fig:svd} for illustrations.

In what follows, we study the dynamics from time $t_{1}$ to the final time $t_{2}=t_{1}+T$, and abbreviate $\mathbf{F}\coloneqq\mathbf{F}_{t_{1}}^{t_{2}}$. In this setting,
studying finite-time dynamics reduces to the (linearized) deformation analysis of a single iteration of a diffeomorphism.

\begin{figure}
\centering
\includegraphics[width=0.9\columnwidth]{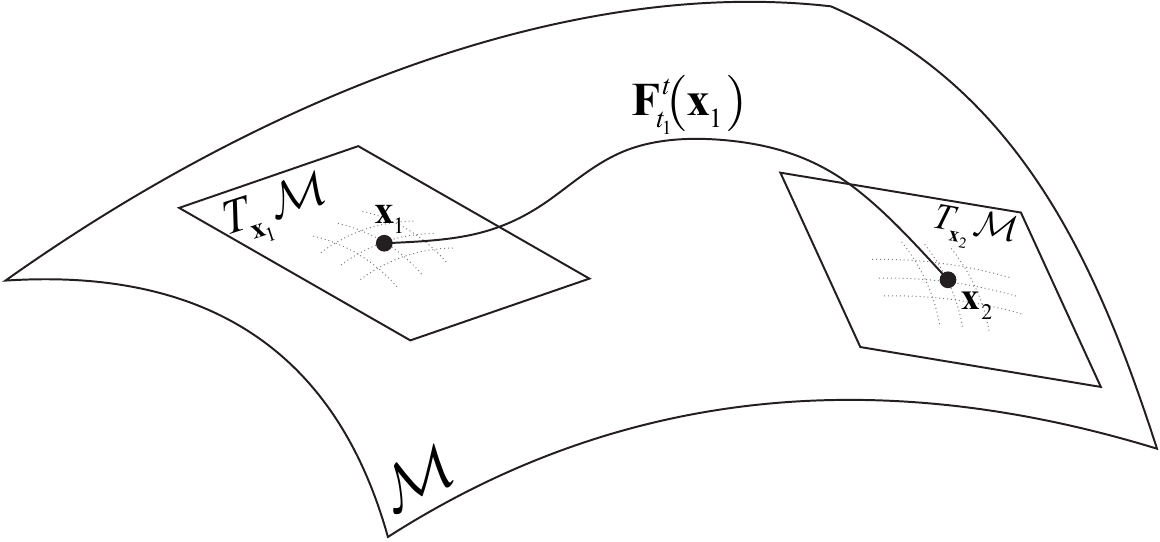}
\caption{The flow geometry.}
\label{fig:flow_tangent}
\end{figure}

\begin{rem}
The fact that the deformation gradient is pointwise a mapping between two different vector spaces precludes formally to pose an eigenvalue problem for $D\mathbf{F}(\mathbf{x}_1)$. When passing to a matrix representation of $D\mathbf{F}(\mathbf{x}_1)$, the solution of the associated eigenvalue problem on $\mathbb{R}^n$ depends on the choice of bases in the two tangent spaces with respect to which the matrix representation is computed. This phenomenon is also paraphrased as \emph{lack of objectivity/frame-invariance}\cite{Mezic2010}.
\end{rem}

The linearized deformation effect of the flow map $\mathbf{F}$ is best studied
with the \emph{singular value decomposition (SVD)} of $D\mathbf{F}$, i.e.,
\[
D\mathbf{F}=\Theta\Sigma\Xi^{\top}.
\]
Here, $\Xi\colon T_\mathcal{D}\mathcal{M}\to T_\mathcal{D}\mathcal{M}$ and $\Theta\colon T_{\mathbf{F}[\mathcal{D}]}\mathcal{M}\to T_{\mathbf{F}[\mathcal{D}]}\mathcal{M}$ are pointwise orthogonal vector bundle homomorphisms, $\Xi^\top$ denotes the pointwise adjoint of $\Xi$ with respect to the metric, and $\Sigma\colon T_\mathcal{D}\mathcal{M}\to T_{\mathbf{F}[\mathcal{D}]}\mathcal{M}$ is a vector bundle isomorphism over $\mathbf{F}$. In orthonormal coordinates, $\Xi$ and $\Theta$ are represented by orthogonal matrices. The columns of $\Xi$, denoted by $\xi_{n},\ldots,\xi_{1}$, are called \emph{right-singular vectors} of $D\mathbf{F}$, and the columns of $\Theta$,  denoted by $\theta_{n},\ldots,\theta_{1}$, are called \emph{left-singular vectors} of $D\mathbf{F}$. Right- and left-singular vectors form pointwise orthonormal bases of the corresponding tangent space. With respect to these bases, $\Sigma$ is represented pointwise by a diagonal matrix with positive entries $\lVert D\mathbf{F}\rVert=\sigma_{n}\geq\ldots\geq\sigma_{1}>0$,
called the \emph{singular values} of $D\mathbf{F}$. In continuum mechanics, one also refers to them as \emph{stretch ratios} (or \emph{principal stretches}) and to the $\xi_{i}$'s and $\theta_{j}$'s as \emph{principal directions} \cite{Truesdell2004}. Note, that we changed the order of indices compared to the usual SVD notation for consistency
with literature related to finite-time Lyapunov analysis, see below.

It is worth recalling that the SVD is well-defined in the following sense, see, e.g.,  Ref.~\onlinecite[Thm.\ 4.1]{Trefethen1997}: requiring
an ordering of the diagonal entries as above makes $\Sigma$ uniquely
defined. Accordingly, an ordering of the normalized singular vectors
is induced, which in turn are uniquely defined up to direction if
the associated singular value is simple. Throughout this work, we
assume the SVD of $D\mathbf{F}$ to be well-defined, and restrict,
if necessary, our analysis to the open subset of the domain in which
this holds true. This assumption guarantees the existence of two sets of (locally) smooth
vector fields: the right-singular vector fields on $\mathcal{D}$, and the left-singular vector fields on $\mathbf{F}[\mathcal{D}]$. We choose to consider the scalar singular value fields $\sigma_{i}$
to be defined on $\mathcal{D}$.

It is readily seen that right- and left-singular
vectors are eigenvectors of right and left \emph{Cauchy--Green strain tensors},
respectively, which are defined as
\begin{equation}\label{eq:CG}
\begin{split}
\mathbf{C} & =\left(D\mathbf{F}\right)^\top D\mathbf{F}=\Xi\Sigma^{2}\Xi^\top,\\
\mathbf{B} & =D\mathbf{F}\left(D\mathbf{F}\right)^\top =\Theta\Sigma^{2}\Theta^\top.
\end{split}
\end{equation}

In contrast to $D\mathbf{F}$, the right and left
Cauchy--Green strain tensor fields are tangent bundle isomorphisms
with base spaces $\mathcal{D}$ and $\mathbf{F}[\mathcal{D}]$, respectively. Therefore, one can
consider fiberwise eigenvalue problems, and Eq.\ \eqref{eq:CG} shows that
right and left Cauchy--Green strain tensors, evaluated at $\mathbf{x}_1$ and $\mathbf{F}(\mathbf{x}_1)$, respectively, have the same eigenvalues
$\lambda_{n}\geq\ldots\geq\lambda_{1}>0$. They are related to the
singular values of $D\mathbf{F}$ via $\lambda_{i}=\sigma_{i}^{2}$.

The SVD can be reformulated
in differential geometric notation as
\begin{equation}
\mathbf{F}_{*}\xi_{i}=\left(\mathbf{F}_{*}\sigma_{i}\right)\theta_{i}.\label{eq:F-relation}
\end{equation}
One also says that $\theta_{i}$ and $\xi_{i}$
are \emph{$\mathbf{F}$-related}, up to the non-vanishing, positive
coefficient function $\sigma_{i}$. Eq.\ \eqref{eq:F-relation} states that the $\xi_i$ are mapped by $D\mathbf{F}$ onto $\theta_i$ and
thereby stretched by $\sigma_i$, see Fig.\ \ref{fig:svd}.

\begin{figure}
\centering
\includegraphics[width=0.9\columnwidth]{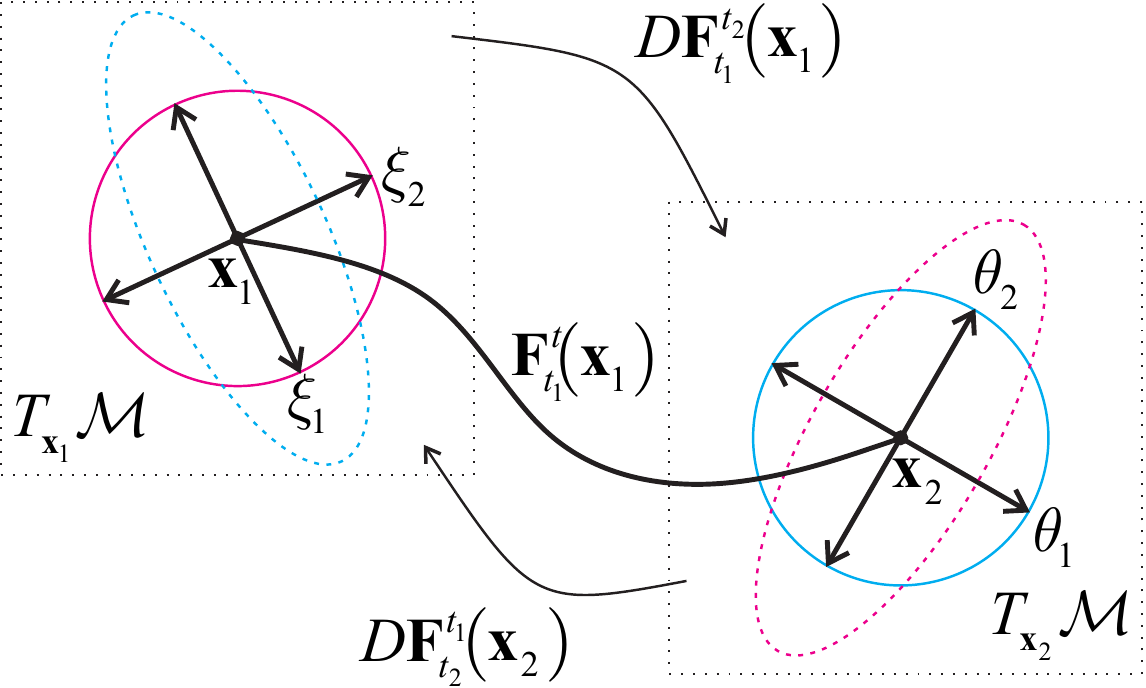}
\caption{The local geometry of linearized deformation in the two-dimensional incompressible case. Solid unit circles are mapped to dashed ellipses of the corresponding color by $D\mathbf{F}$ and $D\mathbf{F}^{-1}$, resp. The black curve connecting $\mathbf{x}_1$ and $\mathbf{x}_2$ represents the trajectory of $\mathbf{x}_1$ under $\mathbf{F}_{t_1}^t$. Due to incompressibility, see Eqs.\ \eqref{eq:svd_incompressible} and \eqref{eq:bw-F-relation}, the lengths of the major and minor semi axes equal each other, respectively.}
\label{fig:svd}
\end{figure}

In words, $\xi_{n}\left(\mathbf{x}_{1}\right)$ and $\xi_{1}\left(\mathbf{x}_{1}\right)$
are the displacements around $\mathbf{x}_{1}$ of, respectively, largest
and smallest stretching under the linearized flow $D\mathbf{F}(\mathbf{x}_{1})$.
At some point $\mathbf{x}_{1}$, the rate of largest stretching over
the time interval $\left[t_{1},t_{2}\right]$ of length $T$ is defined
as
\begin{multline}\label{eq:FTLE}
\Lambda^{\text{f}}\left(\mathbf{x}_{1}\right)\coloneqq\frac{1}{T}\log\lVert D\mathbf{F}\left(\mathbf{x}_{1}\right)\rVert=\\
\frac{1}{T}\log\sigma_{n}\left(\mathbf{x}_{1}\right)=\frac{1}{2T}\log\lambda_{n}\left(\mathbf{x}_{1}\right),
\end{multline}
and is referred to as \emph{(forward) finite-time Lyapunov exponent
(FTLE)}. If $\mathbf{u}$ is incompressible, then the associated flow is volume-preserving, i.e., $\det\left(D\mathbf{F}\right)=\prod_{i=1}^{n}\sigma_{i}=1$,
and therefore $\Lambda^{\text{f}}\geq0$; see, e.g., Ref.~\onlinecite[Proposition 14.20]{Lee2012},
for the precise meaning of $\det(D\mathbf{F})$ in the manifold framework. For two-dimensional incompressible
flows we have in particular
\begin{align}
\sigma_{2} & =\sigma_{1}^{-1}.\label{eq:svd_incompressible}
\end{align}
Here and in the following, the superscript $-1$ at singular value fields
denotes the pointwise reciprocal.

\begin{rem}\label{rem:svd}
It is more common to introduce the FTLE and the principal directions
via the \emph{eigendecomposition} of $\mathbf{C}$ than via the SVD
of $D\mathbf{F}$; see Refs.~\onlinecite{Pobitzer2012,Ma2014} for SVD-based
approaches with a different focus. The SVD-approach has both theoretical
and numerical advantages. First, $D\mathbf{F}$
is the natural object of study and its SVD gives the initial displacements of 
strongest and weakest stretching, together with their final displacements,
see Fig.\ \ref{fig:svd} and \prettyref{rem:svd-bw} for another theoretical aspect.
One numerical advantage is that
the implementation of the SVD guarantees non-negative singular values
and an orthogonal set of right- and left-singular vectors, in contrast to 
general eigendecomposition algorithms. Finally, SVD algorithms
are less sensitive to perturbations than eigenvalue problem solvers\cite{Trefethen1997}.
\end{rem}

As mentioned in the Introduction, the flow map is also considered under backward integration,
either from $t_{1}$ to $t_{1}-T$, or from $t_{1}+T$ to $t_{1}$.
To keep our analysis related to the same time interval $\left[t_{1},t_{2}\right]\coloneqq\left[t_{1},t_{1}+T\right]$\footnote{See also the discussion in Ref.\ \onlinecite{Farazmand2013}},
we continue with our previous choice. The linearized deformation under the backward
flow is then given by
\begin{equation}
D\mathbf{F}^{-1}=\Xi\Sigma^{-1}\Theta^{\top},\label{eq:svd_inverse}
\end{equation}
where we have used $\left(D\mathbf{F}\right)^{-1}=D\mathbf{F}^{-1}$
by the chain rule and the orthogonality of $\Xi$ and $\Theta$. Introducing backward singular value fields $\kappa_n\geq\ldots\geq\kappa_1>0$,
defined on $\mathbf{F}[\mathcal{D}]$, we observe from Eq.\ \eqref{eq:svd_inverse} that the $\theta_1,\ldots,\theta_n$, obtained from the SVD of the forward $D\mathbf{F}$, are exactly the backward principal directions. Furthermore, we have the relations
\begin{equation}
\begin{split}\label{eq:bw-F-relation}
\left(\mathbf{F}^{-1}\right)_{*}\kappa_{n+1-j} &= \mathbf{F}^*\kappa_{n+1-j}=\sigma_{j}^{-1},\\
\left(\mathbf{F}^{-1}\right)_{*}\theta_{j} &= \mathbf{F}^{*}\theta_{j}=\mathbf{F}^*\kappa_{n+1-j}\xi_{j}=\sigma_{j}^{-1}\xi_{j},
\end{split}
\end{equation}
where $\mathbf{F}^{*}$ denotes the pullback by $\mathbf{F}$ \cite{Abraham1988}. Eqs.\ \eqref{eq:bw-F-relation} relate all forward and backward principal stretches and directions to each other. Note the reversed ordering of $\theta$-indices compared to the $\kappa$-indices,
which is due to the fact that we keep the index order induced from the forward dynamics,
see Fig.\ \ref{fig:svd}.

\begin{rem}\label{rem:svd-bw}
In addition to the well-known advantages of SVD mentioned in \prettyref{rem:svd},
another---seemingly overlooked---benefit from introducing deformation terms via
the SVD of $D\mathbf{F}$ is that we obtain forward \emph{and} backward stretch
information from a single forward \emph{or} backward time flow computation.
This fact has been exploited recently for efficient and accurate LCS computations and tracking\cite{Karrasch2013c}.
\end{rem}

As a consequence, the backward FTLE is
\begin{multline}
\Lambda^{\text{b}}\left(\mathbf{x}_{2}\right)=\frac{1}{T}\log\kappa_n(\mathbf{x}_2)=\\
\frac{1}{T}\log\sigma_{1}^{-1}(\mathbf{x}_1)=-\frac{1}{T}\log\sigma_{1}(\mathbf{x}_1),\label{eq:FTLE-backward}
\end{multline}
which was first obtained in Ref.\ \onlinecite[Prop.\ 1]{Haller2011b}, and paraphrased
there as: the largest backward FTLE
equals the negative of the \emph{smallest} forward FTLE $\frac{1}{T}\log\sigma_{1}$.
For two-dimensional incompressible flows, we re-derive Prop.\ 2 in Ref.\ \onlinecite{Haller2011b} from Eq.\ \eqref{eq:svd_incompressible}, i.e.,
\begin{multline}
\Lambda^{\text{b}}\left(\mathbf{x}_{2}\right)=\frac{1}{T}\log\sigma_{1}^{-1}\left(\mathbf{x}_{1}\right)
=\\
\frac{1}{T}\log\sigma_{2}\left(\mathbf{x}_{1}\right)=\Lambda^{\text{f}}\left(\mathbf{x}_{1}\right).
\label{eq:forward-backward}
\end{multline}

Even though Eq.\ \eqref{eq:forward-backward} has been derived earlier\cite{Haller2011b},
a significant implication for visual FTLE ridge-based LCS approaches has been overlooked.
For simplicity, suppose we identify a visual backward FTLE ridge $\mathcal{S}_2$ at the final time, which is backward normally repelling. That is, $\mathcal{S}_2$ is characterized by
high $\Lambda^\textnormal{b}$-values relative to neighboring points off $\mathcal{S}_2$. Its
flow-preimage $\mathcal{S}_1=\mathbf{F}^{-1}(\mathcal{S}_2)$ at the initial time is a forward attracting LCS with high
$\Lambda^{\textnormal{f}}$-values relative to neighboring points off $\mathcal{S}_1$, and thus appears as a visual forward FTLE ridge. Due to backward normal repulsion,
$\mathcal{S}_2$ is backward tangentially shrinking, and consequently $\mathcal{S}_1$ may have
a length on the scale of---or possibly below---computational resolution. Therefore, in practice
the likeliness for a visually observable, attracting FTLE ridge depends on the ratio between the
strength of hyperbolicity, the length of the FTLE ridge and the length of integration
time.
In any case, this does not affect the very fact that strongly forward attracting flow
structures come along with high forward FTLE values, as demonstrated in \prettyref{sec:saddle}.

\section{LCS Approaches Including Principal Directions}\label{sec:LCS}

For most LCS approaches based on particle separation there are no explicit
characterizations of attracting LCSs in forward time.
In this section, we provide a short and direct proof of such a characterization
for the geodesic approach\cite{Farazmand2013}, and the new
corresponding result for the variational approach\cite{Haller2011,Farazmand2012,Karrasch2012}. We revert the chronological order of the development
of these two concepts, as the variational approach contains more mathematical conditions, covering
those of the geodesic approach.

\subsection{Technical preliminaries}

\begin{defn}[{Generalized maximum \cite{Eberly1996}}]
Let $f$ be a smooth scalar field, defined on some open subset $D\subseteq\mathcal{M}$, and $\mathbf{v}$ be a smooth
vector field on $D$. Then $\mathbf{x}\in D$ is a \emph{generalized maximum of $f$ with respect to $\mathbf{v}$}, if
\begin{align*}
\mathcal{L}_{\mathbf{v}}f(\mathbf{x}) & =0, & \mathcal{L}_{\mathbf{v}}^{2}f(\mathbf{x}) &<0.
\end{align*}
Here, $\mathcal{L}_{\mathbf{v}}f$ and $\mathcal{L}_{\mathbf{v}}^{2}f(\mathbf{x})$ denote, respectively,
the first- and second-order Lie derivatives of $f$ with respect to $\mathbf{v}$ at $\mathbf{x}$. A \emph{generalized minimum of $f$} is defined as a generalized maximum of $-f$.
\end{defn}

The first- and second-order Lie derivatives correspond to directional derivatives. In the literature and the Euclidean case, they are also written as
\begin{align*}
\mathcal{L}_{\mathbf{v}}f(\mathbf{x})&=\left\langle \mathbf{v}\left(\mathbf{x}\right),\nabla f(\mathbf{x})\right\rangle,\\
\mathcal{L}_{\mathbf{v}}^{2}f(\mathbf{x}) &= \left\langle \mathbf{v}\left(\mathbf{x}\right),\nabla^{2}f\left(\mathbf{x}\right)\mathbf{v}\left(\mathbf{x}\right)\right\rangle.
\end{align*}

\begin{lem}\label{lem:Lie_transform}
Let $\xi$ and $\theta$ be smooth, normalized vector fields on some open subset $D\subseteq\mathcal{D}\subseteq\mathcal{M}$ and $\mathbf{F}[D]$, respectively,
and $f$ and $g$ be smooth scalar functions on $D$ and $\mathbf{F}[D]$, respectively. Suppose the vector fields and the functions are $\mathbf{F}$-related, i.e.\ $\mathbf{F}_{*}\xi=\left(\mathbf{F}_{*}\sigma\right)\theta$
with a smooth, nowhere vanishing weight function $\sigma$, and $\mathbf{F}_{*}f=g$
or equivalently $\mathbf{F}^{*}g=f$. Then the following equation
holds for any $\mathbf{x}_{1}\in D$ and $\mathbf{x}_{2}=\mathbf{F}(\mathbf{x}_{1})$:
\begin{equation}
\mathcal{L}_{\xi}f\left(\mathbf{x}_{1}\right)=\sigma\left(\mathbf{x}_{1}\right)\mathcal{L}_{\theta}g\left(\mathbf{x}_{2}\right).\label{eq:transfer_rule}
\end{equation}
If \eqref{eq:transfer_rule} vanishes at $\mathbf{x}_{1}$, then
\begin{equation}
\mathcal{L}_{\xi}^{2}f\left(\mathbf{x}_{1}\right)=\sigma^{2}\left(\mathbf{x}_{1}\right)\mathcal{L}_{\theta}^{2}g\left(\mathbf{x}_{2}\right).\label{eq:transfer_rule_2}
\end{equation}
\end{lem}

\begin{proof}
We use \eqref{eq:F-relation}, and Propositions 4.2.8 and 4.2.15(i) from \onlinecite{Abraham1988}, to get
\[
\begin{split}\mathcal{L}_{\xi}f(\mathbf{x}_{1}) & =\mathcal{L}_{\xi}\left(\mathbf{F}^{*}g\right)(\mathbf{x}_{1})=\mathbf{F}^{*}\left(\mathcal{L}_{\left(\mathbf{F}_{*}\sigma\right)\theta}g\right)(\mathbf{x}_{1})\\
 &
=\sigma\left(\mathbf{x}_{1}\right)\mathbf{F}^{*}\mathcal{L}_{\theta}g\left(\mathbf{x}_{1}\right)=\sigma\left(\mathbf{x}_{1}\right)\mathcal{L}_{\theta}g\left(\mathbf{x}_{2}\right).
\end{split}
\]
Eq.\ \eqref{eq:transfer_rule_2} follows similarly:
\begin{align*}
\mathcal{L}_{\xi}^{2}f(\mathbf{x}_{1}) & =\mathcal{L}_{\xi}\left(\mathcal{L}_{\xi}f\right)(\mathbf{x}_{1})\\
 & =\mathcal{L}_{\xi}\left(\sigma\cdot\mathbf{F}^{*}\left(\mathcal{L}_{\theta}g\right)\right)\left(\mathbf{x}_{1}\right)\\
 & =\left(\mathcal{L}_{\xi}\sigma\cdot\mathbf{F}^{*}\left(\mathcal{L}_{\theta}g\right)+\sigma^{2}\cdot\mathbf{F}^{*}\left(\mathcal{L}_{\theta}^{2}g\right)\right)\left(\mathbf{x}_{1}\right)\\
 & =\sigma\left(\mathbf{x}_{1}\right)^{2}\mathcal{L}_{\theta}^{2}g(\mathbf{x}_{2}),
\end{align*}
where the last identity follows from the assumption $\mathcal{L}_{\theta}g\left(\mathbf{x}_{2}\right)=0$.
\end{proof}

\begin{rem}\label{rem:Lie-lemma}
By Eqs.\ \eqref{eq:F-relation} and \eqref{eq:bw-F-relation}, \prettyref{lem:Lie_transform}
applies to Lie derivatives of forward and backward singular value fields with respect to singular vector fields, which will be used to prove \prettyref{thm:variation_characterization}.
\end{rem}

\subsection{The relaxed variational/geodesic approach to hyperbolic LCSs}

In the following, we will work with material surfaces, i.e., flow-invariant, time-parametric
families of codimension-one submanifolds $\mathcal{S}(t)$ of $\mathcal{M}$, $t\in[t_1,t_2]$, satisfying
$\mathbf{F}_s^r(\mathcal{S}(s))=\mathcal{S}(r)$ for any $s,r\in[t_1,t_2]$.

\begin{defn}[Strain- and stretch-surface\cite{Farazmand2013}]
We call a material surface $\mathcal{S}(t)$ a \emph{forward strain-}
and \emph{stretch-surface} if $\mathcal{S}(t_{1})$ is everywhere
normal to the right-singular vector fields $\xi_{n}$ and $\xi_{1}$, respectively,
i.e., for any $\mathbf{x}_{1}\in\mathcal{S}(t_{1})$ one has, respectively,
\begin{align*}
\xi_{n}(\mathbf{x}_{1}) & \perp T_{\mathbf{x}_{1}}\mathcal{S}(t_{1}), & \text{\text{and}} &  & \xi_{1}(\mathbf{x}_{1}) & \perp T_{\mathbf{x}_{1}}\mathcal{S}(t_{1}).
\end{align*}
We call $\mathcal{S}(t)$ a \emph{backward strain-} and \emph{stretch-surface} if $\mathcal{S}(t_{2})$ is everywhere normal to the left-singular vector fields
$\theta_{1}$ and $\theta_{n}$, respectively, i.e., for any $\mathbf{x}_{2}\in\mathcal{S}(t_{2})$
one has, respectively,
\begin{align*}
\theta_{1}(\mathbf{x}_{2}) & \perp T_{\mathbf{x}_{2}}\mathcal{S}(t_{2}), & \text{and} &  & \theta_{n}(\mathbf{x}_{2}) & \perp T_{\mathbf{x}_{2}}\mathcal{S}(t_{2}).
\end{align*}
\end{defn}

Strain-surfaces are normally repelling flow structures, while stretch-surfaces
are normally attracting. Both types admit no Lagrangian shear\cite{Haller2012}. In dimension $2$, strain- and stretch-surfaces are referred to as \emph{strain-} and \emph{stretchlines}. They have the \emph{local} property, that they are pointwise aligned with the direction of strongest compression or stretching, respectively.
In Ref.\ \onlinecite{Farazmand2014a} it is shown that they also satisfy a \emph{global} variational principle, and that they are null-geodesics of a Lorentzian metric. In dimension $3$ and higher, strain- and stretch-surfaces arise from a local variational principle, which compares candidate material surfaces to those perturbed in tangent orientation pointwise. Existence of (hyper-)surfaces orthogonal
to a given (nonlinear) vector field as well as their numerical construction
are challenging\cite{Abraham1988,Palmerius2009}, see Ref.~\onlinecite{Blazevski2014} for the local variational approach in dimension $3$.


\begin{thm}[{\onlinecite[Theorem 1]{Farazmand2013}}]\label{thm:forward_backward}
Forward strain-surfaces coincide with backward stretch-sur\-faces, and forward stretch-surfaces coincide with backward strain-surfaces.
\end{thm}

\begin{proof}
We prove only that backward strain-surfaces coincide with forward stretch-sur\-faces, the other statement can be shown analogously. The relation $\theta_1(\mathbf{x}_{2})\perp T_{\mathbf{x}_{2}}\mathcal{S}(t_{2})$ is equivalent to the fact that $T_{\mathbf{x}_{2}}\mathcal{S}(t_{2})$ is the span of $\left\lbrace\theta_{2}(\mathbf{x}_{2}),\ldots,\theta_{n}\left(\mathbf{x}_{2}\right)\right\rbrace$.
$T_{\mathbf{x}_{2}}\mathcal{S}(t_{2})$ is mapped bijectively to $T_{\mathbf{x}_{1}}\mathcal{S}\left(t_{1}\right)$
by $D\mathbf{F}^{-1}$ by flow-invariance of $\mathcal{S}(t)$. In turn,
$T_{\mathbf{x}_{1}}\mathcal{S}\left(t_{1}\right)$ is spanned by $\left\lbrace\xi_{2}(\mathbf{x}_{1}),\ldots,\xi_{n}\left(\mathbf{x}_{1}\right)\right\rbrace$
and consequently orthogonal to $\xi_{1}\left(\mathbf{x}_{1}\right)$.
\end{proof}

\begin{example}\label{example:saddle_geo}
We apply the geodesic LCS approach to the saddle flow from
\prettyref{sec:saddle}. Here, we only focus on checking the common hypothesis, 
that FTLE ridges computed in forward time correspond to repelling LCSs. 
in Fig.\ \ref{fig:nlin_saddle_geo}, we plot strainlines (red), which are normally repelling and do not pick up the visual FTLE ridge. In contrast, stretchlines, which are normally 
attracting, do align with the orientation of the visual ridge, which is 
picked up by the one passing through the saddle at the origin.

\begin{figure}
\centering
\includegraphics[width=0.6\columnwidth]{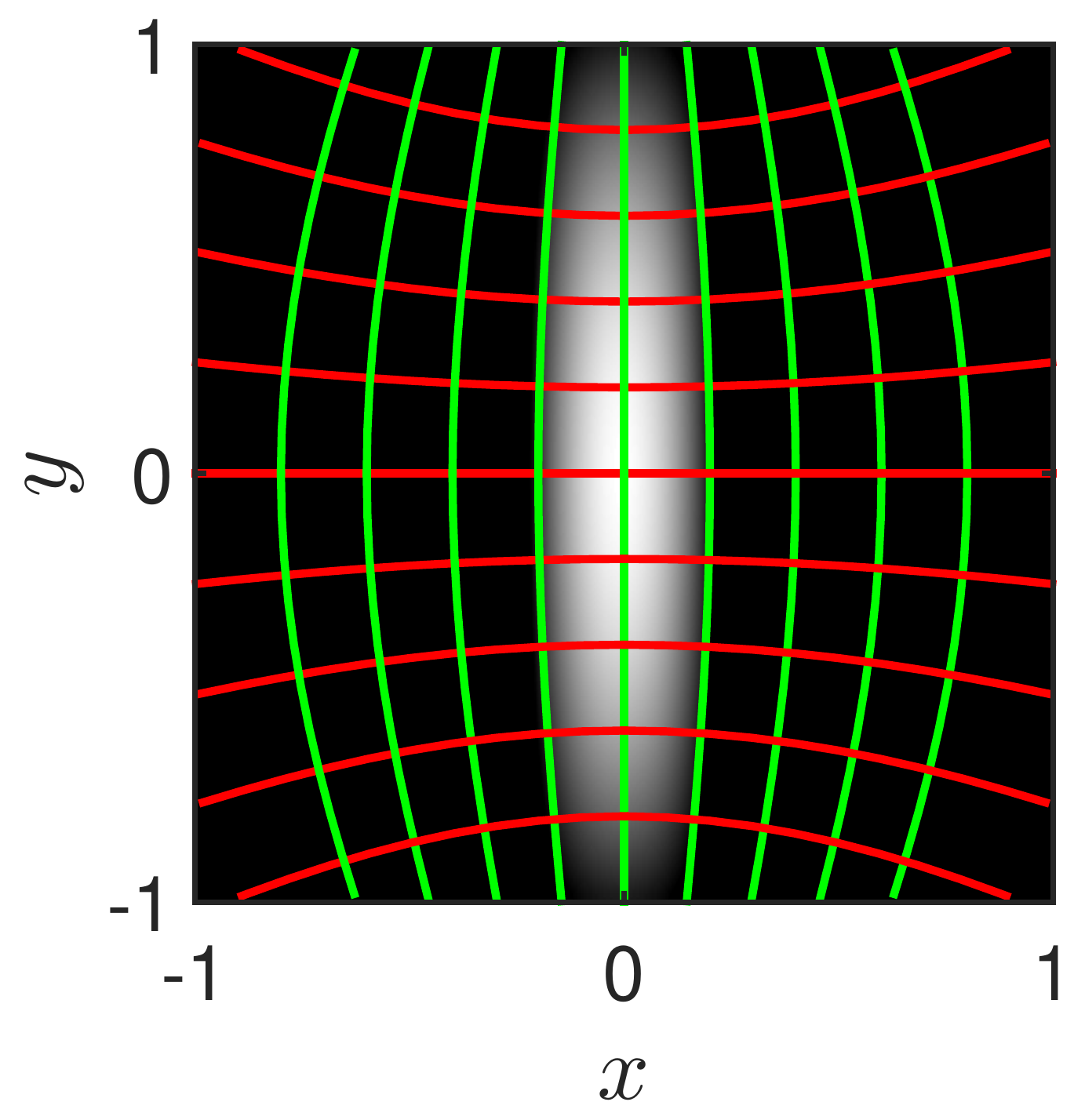}
\caption{Strainlines (red) and stretchlines (green) on top of the forward FTLE field from Fig.\ \ref{fig:nlin_saddle}(b) in gray scale. Clearly, only normally attracting structures align with the visual FTLE ridge.}
\label{fig:nlin_saddle_geo}
\end{figure}
\end{example}

\subsection{The variational approach to hyperbolic LCSs}

In Ref.~\onlinecite{Haller2011}, hyperbolic LCSs are defined variationally as locally
strongest normally repelling material lines.
The local variational principle compares LCS candidate material surfaces
to surfaces smoothly perturbed in position and tangent orientation.

\begin{defn}[{Variational hyperbolic LCS\cite{Haller2011,Karrasch2012}}]\label{def:varLCS}
We call a compact material hypersurface $\mathcal{S}(t)$ a \emph{variational
repelling LCS over $[t_{1},t_{2}]$} if
\begin{enumerate}
\item $\mathcal{S}(t)$ is a forward strain-surface, i.e.\ $\xi_n\perp T\mathcal{S}(t_1)$,
\item each $\mathbf{x}_1\in\mathcal{S}(t_1)$ is a generalized maximum of $\sigma_n$ with respect to $\xi_n$, i.e.\ $\mathcal{L}_{\xi_n}\sigma_n=0$ and $\mathcal{L}^2_{\xi_n}\sigma_n<0$ on $\mathcal{S}(t_1)$, and
\item at each $\mathbf{x}_1\in\mathcal{S}(t_1)$ one has $\sigma_{n-1}(\mathbf{x}_{1})\neq\sigma_{n}(\mathbf{x}_{1})>1$.
\end{enumerate}
We call $\mathcal{S}(t)$ a \emph{variational attracting LCS over
$[t_{1},t_{2}]$} if it is a variational repelling LCS over $\left[t_{1},t_{2}\right]$
for the backward flow.
\end{defn}
Note that the differentiability assumption on singular vector fields\cite{Karrasch2012}
is satisfied by our general hypothesis on simplicity of singular values. The third condition in
\prettyref{def:varLCS} requests that normal repulsion dominates tangential stretching.

\begin{thm}[Characterization of variational attracting LCSs]
\label{thm:variation_characterization} Assume that $\mathcal{S}(t)$
is a variational attracting LCS over $[t_{1},t_{2}]$. Then the following statements hold:
\begin{enumerate}
\item $\mathcal{S}(t)$ is a forward stretch-surface,
\item each $\mathbf{x}_1\in\mathcal{S}(t_1)$ is a generalized minimum of $\sigma_1$ with respect to $\xi_1$, and
\item at each $\mathbf{x}_1\in\mathcal{S}(t_1)$ one has $\sigma_{2}(\mathbf{x}_{1})\neq\sigma_{1}(\mathbf{x}_{1})<1$.
\end{enumerate}
\end{thm}

\begin{proof}
Item 1.\ is the statement of \prettyref{thm:forward_backward}.
Item 2.\ follows from \prettyref{lem:Lie_transform}, cf.\ \prettyref{rem:Lie-lemma}, and strict monotonic decay of the inversion on $\R_{>0}$. Finally, $\sigma_{2}(\mathbf{x}_{1})\neq\sigma_{1}(\mathbf{x}_{1})<1$ is clearly equivalent to the assumption $\kappa_{n-1}(\mathbf{x}_{2})=\sigma_{2}^{-1}(\mathbf{x}_{1})\neq\sigma_{1}^{-1}(\mathbf{x}_{1})=\kappa_n(\mathbf{x}_{2})>1$.
\end{proof}

\begin{cor}
\label{cor:variation_characterization}Let $\mathcal{S}(t)$ be a
variational attracting LCS over $[t_{1},t_{2}]$ for a two-dimensional
incompressible velocity field. Then each $\mathbf{x}_{1}\in\mathcal{S}(t_{1})$
satisfies in particular:
\begin{gather*}
\mathcal{L}_{\xi_{1}}\sigma_{2}(\mathbf{x}_{1})=0,\\
\mathcal{L}_{\xi_{1}}^{2}\sigma_{2}(\mathbf{x}_{1})<0,\\
\xi_{1}(\mathbf{x}_{1})\perp T_{\mathbf{x}_{1}}\mathcal{S}(t_{1}).
\end{gather*}
\end{cor}

\prettyref{cor:variation_characterization} states that variational attracting
LCSs appear as curves of generalized maxima of the forward FTLE field,
just as the variational repelling LCS. The distinction between these
two types is via the direction field of differentiation. For attracting variational LCSs,
the direction of largest stretching is tangential to the curve, as
the direction of strongest attraction is normal to it. In contrast
to the topological ridge approaches such as the height ridge,
the variational approach, roughly speaking,
selects ridges according to their orientation with respect to the
directions of strongest attraction and repulsion, respectively.

\begin{example}
We apply the variational LCS approach to the saddle flow from \prettyref{sec:saddle}. In Fig.\ \ref{fig:nlin_saddle_genmax}, we show curves of generalized maxima (red) of the forward FTLE w.r.t.\  $\xi_2$ (red) and w.r.t.\ $\xi_1$ (green). Irrespective of the orientation of the visual FTLE ridge, (parts of) the $x$- and $y$-axes are correctly detected as repelling and attracting LCSs, respectively. The vertical curves of generalized maxima of $\sigma_2$ w.r.t.\ $\xi_2$ in Fig.\ \ref{fig:nlin_saddle_genmax} are discarded by checking for orthogonality with respect to $\xi_2$ (or, equivalently, for tangency with $\xi_1$). In general flows, this orthogonality requirement must be relaxed for numerical stability\cite{Farazmand2012}.
\end{example}

\begin{figure}
\centering
\includegraphics[width=0.6\columnwidth]{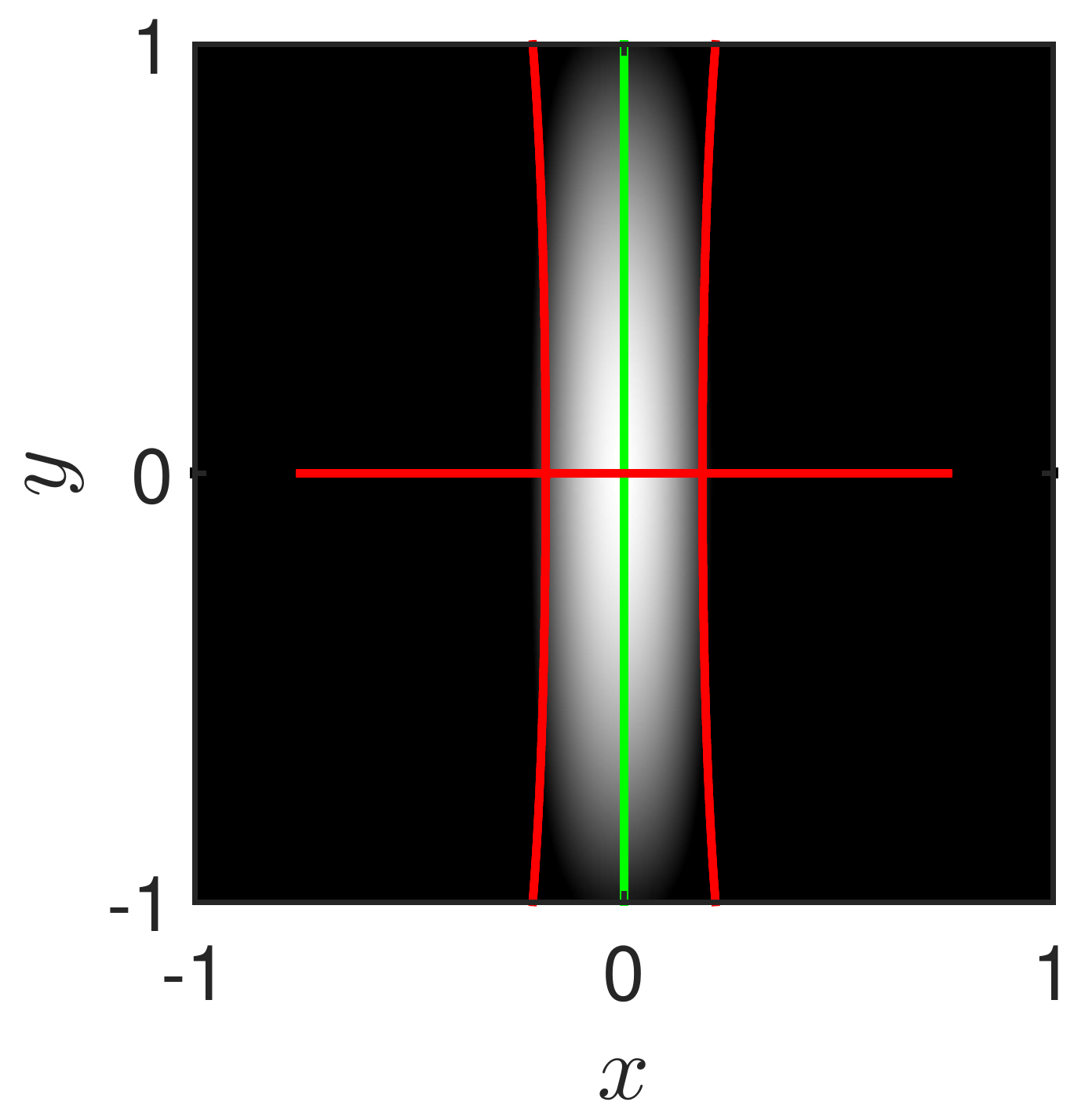}
\caption{Variational LCS candidates on top of the forward FTLE field from Fig.\ \ref{fig:nlin_saddle}(b) in gray scale. Generalized maxima of $\sigma_2$ w.r.t.\ $\xi_2$ (red), i.e., repelling LCS candidates; generalized maxima of $\sigma_2$ w.r.t.\ $\xi_1$ (green), i.e., attracting LCS candidates. Clearly, only normally attracting structures align with the visual FTLE ridge. For the full variational LCS definition, red/green lines shown here must be orthogonal to green/red lines in Fig.\ \ref{fig:nlin_saddle_geo}. Therefore, the two vertical red lines should be discarded.}
\label{fig:nlin_saddle_genmax}
\end{figure}

\section{Conclusions}

In this article, we showed that defining attracting and repelling LCSs
as the ridges of scalar separation measure fields, computed in backward and forward time,
respectively, is mathematically inconsistent. This adds to a collection of
previously discovered issues in FTLE ridge-based LCS detection \cite{Haller2002,Branicki2010,Haller2011,Norgard2012}.

The reason for that inconsistency is the fact that separation measures
reflect separation in an undirected fashion. The question, whether
or not a certain structure of interest, such as a ridge, aligns with
a certain direction field, can be answered only after the computation
of the separation measure field and after the selection of the
structure. In two-dimensional volume-preserving flows, the pointwise 
balance of particle separation and compression complicates the
situation further. In this case, normally attracting
material lines come along with large particle separation and may appear
as ridges in separation measure fields.

This conceptual issue cannot be resolved by refinements in ridge notions, more accurate
computations of separation measure fields, or improvements in
numerical ridge extraction, but only by more involved LCS concepts.
A natural extension of purely ridge-based LCS approaches is to
incorporate principal directions
in order to get a self-consistent notion of hyperbolic LCSs. In
the recently developed geodesic approach \cite{Haller2012,Farazmand2012a,Blazevski2014,Farazmand2014a,Karrasch2013c},
the principal direction fields are used to first define candidate
structures. Only in a second step those structures are selected,
which have the locally strongest intended dynamical impact. This
way, normal repulsion and normal attraction are guaranteed from the
beginning, while the second step guarantees the relevance for explaining
finite-time pattern formation and transport.

\begin{acknowledgments}
I would like to thank George Haller for support and
helpful comments. Furthermore,
I am grateful to Dan Bla\-zev\-ski, Florian Huhn, and especially 
Simon Eugster for fruitful discussions, to Joseph Lydon for proofreading 
the manuscript and to Mohammad Farazmand for sharing his strainline integration code. 
I am also thankful to Clancy Rowley for pointing out Ref.~\onlinecite{Trefethen1997}.
\end{acknowledgments}

\appendix

\section{Representation aspects in linearized deformation}\label{sec:representation}

In this section, we recall how to derive the metric representation of the deformation gradient, when $\mathcal{M}$ is embedded in, say, Euclidean three-space, and the metric on $\mathcal{M}$ is pulled back from that ambient space. As an example, we consider the two-dimensional sphere parametrized by geographical coordinates. Textbook references on the subject include Refs.~\onlinecite{Lee2012,Abraham1988}, see also Ref.~\onlinecite{Lekien2010} for aspects on deformation gradient estimation on triangulated manifolds.

\subsection{Preliminaries}

Let $\left(\mathbb{E},\left\langle \cdot,\cdot\right\rangle \right)$
be some $n$-dimensional Euclidean vector space, and $\mathcal{B}=\left(b_{1},\ldots,b_{n}\right)$
a basis. Then the matrix $G^{\mathcal{B}}=\left(g_{ij}\right)_{ij}\in\mathbb{R}^{n\times n}$,
$g_{ij}=\left\langle b_{i},b_{j}\right\rangle $, called the \emph{Gramian},
is the \emph{metric representing matrix}, which means the following.

Let $v\in\mathbb{E}$ and $v=\sum_{i=1}^{n}v_{i}b_{i}$ the unique
representation of $v$ w.r.t. the basis $\mathcal{B}$. Then the $n$-tuple $v^{\mathcal{B}}=\begin{pmatrix}v_{1} & \cdots & v_{n}\end{pmatrix}^{\top}$ is the \emph{coordinate representation} of $v$ w.r.t.\ $\mathcal{B}$.
Let $v,w\in\mathbb{E}$ with coordinate representations $v^{\mathcal{B}}$
and $w^{\mathcal{B}}$. Then the inner product of $v$ and $w$ is
given by
\[
\left\langle v,w\right\rangle =\left(v^{\mathcal{B}}\right)^{\top}\cdot G^{\mathcal{B}}\cdot w^{\mathcal{B}},
\]
where the right-hand side is meant in terms of the usual matrix-multiplication.
Note that the left-hand side is defined independently from any basis,
but the right-hand side gives a recipe how to actually compute the
inner product in coordinates.

The transformation given by the square root $\left(G^{\mathcal{B}}\right)^{\frac{1}{2}}$ of
$G^{\mathcal{B}}$ yields a basis $\mathcal{B}'=\left(b_{1}',\ldots,b_{n}'\right)$, $b_i'=\left(G^{\mathcal{B}}\right)^{\frac{1}{2}}b_{i}$, in which the inner product can be calculated in Euclidean fashion (recall that if a symmetric, (semi-)positive definite operator $G$ is given as $G=L^{\top}L$, then the square root of $G$ is also referred to as the \emph{modulus of
$L$}, and denoted by $\left|L\right|$). Indeed,
\begin{equation}
\begin{split}\left(v^{\mathcal{B}'}\right)^{\top}\cdot w^{\mathcal{B}'} &=\left(\left(G^{\mathcal{B}}\right)^{\frac{1}{2}}v^{\mathcal{B}}\right)^{\top}\cdot\left(G^{\mathcal{B}}\right)^{\frac{1}{2}}w^{\mathcal{B}}\\
&=\left(v^{\mathcal{B}}\right)^{\top}\cdot G^{\mathcal{B}}\cdot w^{\mathcal{B}}=\left\langle v,w\right\rangle.
\end{split}
\label{eq:metric_coord}
\end{equation}
Because of Eq. \eqref{eq:metric_coord}, we refer to $\left(G^{\mathcal{B}}\right)^{\frac{1}{2}}$
as the \emph{coordinate transformation to metric coordinates}.

\subsection{Pullback metrics}

For convenience, we consider in the following exclusively $\mathbb{E}=\mathbb{E}^{3}$, i.e.,
the Euclidean three-space with canonical inner product. We parametrize
$\mathbb{E}$ by the canonical embedding $\mathbb{R}^{3}\hookrightarrow\mathbb{E}$,
and we refer to the three coordinate directions as $x,y,z$. In these
coordinates, the Riemannian metric is represented by
\[
g_{\mathbb{E}}= 1\mathrm{d} x^{2}+1\mathrm{d} y^{2}+1\mathrm{d} z^{2},
\]
or, in other words, the metric representing matrix w.r.t.\ the canonical
orthonormal dual basis $\left(\mathrm{d} x,\mathrm{d} y,\mathrm{d} z\right)$ is the identity
matrix $I$ at each point.

Let $\mathcal{M}$ be an embedded submanifold of $\mathbb{E}$, and
assume for simplicity that (almost all of) $\mathcal{M}$ can be described
by a single parametrization $P\colon D\subseteq\mathbb{R}^{m}\to\mathcal{M}\hookrightarrow\mathbb{E}$,
$m\leq3$. Examples include the two-dimensional sphere $\mathcal{S}_R^{2}$ of radius $R$ in geographical coordinates, i.e., \emph{longitude} and \emph{latitude}, and a
three-dimensional spherical shell $\mathcal{S}^{2}\times\left[r,R\right]$,
$0<r<R$. For definiteness, we consider only two-dimensional manifolds in the following.

With the parametrization $P=(P^x,P^y,P^z)^\top$, depending on parameters $\mathbf{p}=(p_1,p_2)$, and given in terms of $(x,y,z)$-coordinates, the pushforward of vector fields in $\left(\partial_{p_1},\partial_{p_2}\right)$-$\left(\partial_{x},\partial_{y},\partial_{z}\right)$-coordinates by $P$, is represented by
\[
P_{*}(\mathbf{p})=P'(\mathbf{p})=
\begin{pmatrix}
\partial_{p_1}P^x(\mathbf{p}) & \partial_{p_2}P^x(\mathbf{p})\\
\partial_{p_1}P^y(\mathbf{p}) & \partial_{p_2}P^y(\mathbf{p})\\
\partial_{p_1}P^z(\mathbf{p}) & \partial_{p_2}P^z(\mathbf{p})
\end{pmatrix}.
\]
The pullback $P^{*}$ of differential forms by $P$ (in the dual $\left(\mathrm{d} x,\mathrm{d} y,\mathrm{d} z\right)$-$\left(\mathrm{d}p_1,\mathrm{d}p_2\right)$-co\-or\-di\-nates),
as the dual mapping to the pushforward, is then represented by $P^{*}(\mathbf{p})=P'(\mathbf{p})^\top$.

We turn $\mathcal{M}$ into a Riemannian manifold, by pulling back
the Euclidean metric $g_{\mathbb{E}}$ by $P$, i.e.\ $g_{\mathcal{M}}=P^{*}g_{\mathbb{E}}$.
To compute $g_{\mathcal{M}}$ in local coordinates, it suffices to compute the pullbacks of the canonical differential
forms, for instance, $P^*\mathrm{d}x$. Their coordinates correspond to the respective column in the
matrix representation of $P^{*}$. A closer inspection of the emerging coefficients reveals that
the metric representing matrix of $g_{\mathcal{M}}$ can be computed formally
as $(P_{*})^{\top}P_{*}$. Finally, the coordinate transformation to metric coordinates on the tangent space is given by $\left|P_{*}\right|$.

\subsection{Metric representations of deformation gradients}

Let $\mathbf{F}\colon \mathcal{D}\subseteq\mathcal{M}\to F[\mathcal{D}]\subseteq\mathcal{M}$ be the flow map as introduced in \prettyref{sec:Geometry}. Consider the velocity field $\mathbf{u}$ in Eq.\ \eqref{eq:ODE} to be given in parameters, say, in longitude-latitude on the sphere, and $\mathbf{F}$ to be computed in parameters as well. Then the Jacobian of $D\mathbf{F}$ in these coordinates is computed as
\[
D\mathbf{F}(\mathbf{x}_1)=
\begin{pmatrix}
\partial_{p_1}\mathbf{F}^{p_1}(\mathbf{x}_1) & \partial_{p_2}\mathbf{F}^{p_1}(\mathbf{x}_1)\\
\partial_{p_1}\mathbf{F}^{p_2}(\mathbf{x}_1) & \partial_{p_2}\mathbf{F}^{p_2}(\mathbf{x}_1)
\end{pmatrix}
\]

By the Courant--Fischer Min--Max Theorem\cite{Allaire2008}, singular values and vectors of $D\mathbf{F}$ can be characterized as solutions of min-max-problems of
\begin{equation}
\frac{\left\lVert D\mathbf{F}(\mathbf{x}_1)\mathbf{v}\right\rVert}{\left\lVert\mathbf{v}\right\rVert},\label{eq:deform}
\end{equation}
i.e., the deformation effect of $\mathbf{F}$ at $\mathbf{x}_1\in\mathcal{M}$ on a tangent vector $\mathbf{v}\in T_{\mathbf{x}_1}\mathcal{M}\setminus\left\{ 0\right\}$.
As usual, the norm $\lVert\mathbf{v}\rVert$ of a tangent vector $\mathbf{v}\in T_{\mathbf{x}_1}\mathcal{M}$ is
given by $g_{\mathcal{M},\mathbf{x}_1}(\mathbf{v},\mathbf{v})^{\frac{1}{2}}$. The question now is to which matrix representation to apply an \texttt{svd}-routine.

To this end, we transform the tangent vectors in Eq.\ \eqref{eq:deform} to metric coordinates, which yields:
\begin{multline}\label{eq:deform_metric}
\frac{\left\lVert D\mathbf{F}(\mathbf{x}_1)\mathbf{v}\right\rVert}{\left\lVert\mathbf{v}\right\rVert}
=\frac{\left\lVert\left|P_{*}(\mathbf{x}_2)\right|D\mathbf{F}(\mathbf{x}_1)\mathbf{v}\right\rVert}{\left\lVert\left|P_{*}(\mathbf{x}_1)\right|\mathbf{v}\right\rVert}\\
=\frac{\left\lVert\left|P_{*}(\mathbf{x}_2)\right|D\mathbf{F}(\mathbf{x}_1)\left|P_{*}(\mathbf{x}_1)\right|^{-1}\mathbf{w}\right\rVert_2}{\left\lVert\mathbf{w}\right\rVert_2},
\end{multline}
where $\left\lVert\cdot\right\rVert_{2}$ is the Euclidean norm of tuples. For
the second identity, we substituted $\mathbf{w}=\left|P_{*}(\mathbf{x}_1)\right|\mathbf{v}$, or,
equivalently, $\mathbf{v}=\left|P_{*}(\mathbf{x}_1)\right|^{-1}\mathbf{w}$. Eq.\ \eqref{eq:deform_metric} states that
\[
\left|P_{*}(\mathbf{x}_2)\right|D\mathbf{F}(\mathbf{x}_1)\left|P_{*}(\mathbf{x}_1)\right|^{-1}
\]
is the metric representation of the deformation gradient. This can be used as input for an \texttt{svd}-routine. Finally, right- and left-singular vectors thus obtained need to be transformed to natural coordinates by $\left|P_{*}(\mathbf{x}_1)\right|^{-1}$ and $\left|P_{*}(\mathbf{x}_2)\right|^{-1}$, respectively. After this step, they can be integrated to obtain strain- and stretchlines in the same parametrization as were given $\mathbf{u}$ and $\mathbf{F}$ initially.

\subsection{Example}

Consider the two-dimensional sphere of radius $R$, $\mathcal{S}^2_R$, without the poles, with geographical parametrization given by
\[
P\colon\left(\phi,\theta\right)\mapsto P(\phi,\theta)=\begin{pmatrix}P^{x}(\phi,\theta)\\
P^{y}(\phi,\theta)\\
P^{z}(\phi,\theta)
\end{pmatrix}=\begin{pmatrix}R\cos\theta\cos\phi\\
R\cos\theta\sin\phi\\
R\sin\theta
\end{pmatrix}.
\]
Then the pushforward of vector fields by $P$ (in $\left(\partial_{\phi},\partial_{\theta}\right)$-$\left(\partial_{x},\partial_{y},\partial_{z}\right)$-coordinates)
is represented by
\[
P_{*}(\phi,\theta)=P'(\phi,\theta)=R\begin{pmatrix}-\cos\theta\sin\phi & -\sin\theta\cos\phi\\
-\cos\theta\cos\phi & -\sin\theta\sin\phi\\
0 & \cos\theta
\end{pmatrix}.
\]
Thus, the metric representing matrix has the form
\[
G(\phi,\theta)=P'(\phi,\theta)^\top P'(\phi,\theta)=
\begin{pmatrix}R^{2}\cos^{2}\theta & 0\\
0 & R^{2}
\end{pmatrix}.
\]
The diagonal form is a consequence of the orthogonality of the geographical parametrization, the deviation from the identity matrix is due to the
lack of normalization. Clearly, the coordinate transformation $\left|P_{*}\right|$
to metric coordinates in the tangent space is then given by
\begin{align*}
\left|P_{*}(\phi,\theta)\right|&=\begin{pmatrix}R\cos\theta & 0\\
0 & R
\end{pmatrix},
\intertext{with inverse}
\left|P_{*}(\phi,\theta)\right|^{-1}&=\begin{pmatrix}\frac{1}{R\cos\theta} & 0\\
0 & \frac{1}{R}
\end{pmatrix}.
\end{align*}
Finally, the metric representation of the deformation gradient at $\mathbf{x}_1$, with coordinates $(\phi_1,\theta_1)$, where $\mathbf{x}_2=\mathbf{F}(\mathbf{x}_1)$, with coordinates $(\phi_2,\theta_2)$, takes the form
\begin{multline}
\left|P_{*}(\mathbf{x}_2)\right|D\mathbf{F}(\mathbf{x}_1)\left|P_{*}(\mathbf{x}_1)\right|^{-1}=\\
=\begin{pmatrix}
\frac{\cos\theta_{2}}{\cos\theta_{1}}\partial_{\phi}\mathbf{F}^{\phi}(\mathbf{p}) & \cos\theta_{2}\partial_{\theta}\mathbf{F}^{\phi}(\mathbf{p})\\
\frac{1}{\cos\theta_{1}}\partial_{\phi}\mathbf{F}^{\theta}(\mathbf{p}) &
\partial_{\theta}\mathbf{F}^{\theta}(\mathbf{p})
\end{pmatrix}.
\end{multline}


\begin{thebibliography}{37}%
\makeatletter
\providecommand \@ifxundefined [1]{%
 \@ifx{#1\undefined}
}%
\providecommand \@ifnum [1]{%
 \ifnum #1\expandafter \@firstoftwo
 \else \expandafter \@secondoftwo
 \fi
}%
\providecommand \@ifx [1]{%
 \ifx #1\expandafter \@firstoftwo
 \else \expandafter \@secondoftwo
 \fi
}%
\providecommand \natexlab [1]{#1}%
\providecommand \enquote  [1]{``#1''}%
\providecommand \bibnamefont  [1]{#1}%
\providecommand \bibfnamefont [1]{#1}%
\providecommand \citenamefont [1]{#1}%
\providecommand \href@noop [0]{\@secondoftwo}%
\providecommand \href [0]{\begingroup \@sanitize@url \@href}%
\providecommand \@href[1]{\@@startlink{#1}\@@href}%
\providecommand \@@href[1]{\endgroup#1\@@endlink}%
\providecommand \@sanitize@url [0]{\catcode `\\12\catcode `\$12\catcode
  `\&12\catcode `\#12\catcode `\^12\catcode `\_12\catcode `\%12\relax}%
\providecommand \@@startlink[1]{}%
\providecommand \@@endlink[0]{}%
\providecommand \url  [0]{\begingroup\@sanitize@url \@url }%
\providecommand \@url [1]{\endgroup\@href {#1}{\urlprefix }}%
\providecommand \urlprefix  [0]{URL }%
\providecommand \Eprint [0]{\href }%
\providecommand \doibase [0]{http://dx.doi.org/}%
\providecommand \selectlanguage [0]{\@gobble}%
\providecommand \bibinfo  [0]{\@secondoftwo}%
\providecommand \bibfield  [0]{\@secondoftwo}%
\providecommand \translation [1]{[#1]}%
\providecommand \BibitemOpen [0]{}%
\providecommand \bibitemStop [0]{}%
\providecommand \bibitemNoStop [0]{.\EOS\space}%
\providecommand \EOS [0]{\spacefactor3000\relax}%
\providecommand \BibitemShut  [1]{\csname bibitem#1\endcsname}%
\let\auto@bib@innerbib\@empty
\bibitem [{\citenamefont {{H}aller}\ and\ \citenamefont
  {{S}apsis}(2011)}]{Haller2011b}%
  \BibitemOpen
  \bibfield  {author} {\bibinfo {author} {\bibfnamefont {G.}~\bibnamefont
  {{H}aller}}\ and\ \bibinfo {author} {\bibfnamefont {T.}~\bibnamefont
  {{S}apsis}},\ }\bibfield  {title} {\enquote {\bibinfo {title} {{L}agrangian
  coherent structures and the smallest finite-time {Lyapunov} exponent},}\
  }\href {\doibase 10.1063/1.3579597} {\bibfield  {journal} {\bibinfo
  {journal} {Chaos}\ }\textbf {\bibinfo {volume} {21}},\ \bibinfo {pages}
  {023115} (\bibinfo {year} {2011})}\BibitemShut {NoStop}%
\bibitem [{\citenamefont {{F}arazmand}\ and\ \citenamefont
  {{H}aller}(2013)}]{Farazmand2013}%
  \BibitemOpen
  \bibfield  {author} {\bibinfo {author} {\bibfnamefont {M.}~\bibnamefont
  {{F}arazmand}}\ and\ \bibinfo {author} {\bibfnamefont {G.}~\bibnamefont
  {{H}aller}},\ }\bibfield  {title} {\enquote {\bibinfo {title} {{A}ttracting
  and repelling {L}agrangian coherent structures from a single computation},}\
  }\href {\doibase 10.1063/1.4800210} {\bibfield  {journal} {\bibinfo
  {journal} {Chaos}\ }\textbf {\bibinfo {volume} {23}},\ \bibinfo {pages}
  {023101} (\bibinfo {year} {2013})}\BibitemShut {NoStop}%
\bibitem [{\citenamefont {{H}aller}(2011)}]{Haller2011}%
  \BibitemOpen
  \bibfield  {author} {\bibinfo {author} {\bibfnamefont {G.}~\bibnamefont
  {{H}aller}},\ }\bibfield  {title} {\enquote {\bibinfo {title} {{A}
  variational theory of hyperbolic {L}agrangian {C}oherent {S}tructures},}\
  }\href {\doibase 10.1016/j.physd.2010.11.010} {\bibfield  {journal} {\bibinfo
   {journal} {Physica D}\ }\textbf {\bibinfo {volume} {240}},\ \bibinfo {pages}
  {574--598} (\bibinfo {year} {2011})}\BibitemShut {NoStop}%
\bibitem [{\citenamefont {{H}aller}\ and\ \citenamefont
  {{Y}uan}(2000)}]{Haller2000}%
  \BibitemOpen
  \bibfield  {author} {\bibinfo {author} {\bibfnamefont {G.}~\bibnamefont
  {{H}aller}}\ and\ \bibinfo {author} {\bibfnamefont {G.}~\bibnamefont
  {{Y}uan}},\ }\bibfield  {title} {\enquote {\bibinfo {title} {{L}agrangian
  coherent structures and mixing in two-dimensional turbulence},}\ }\href
  {\doibase 10.1016/S0167-2789(00)00142-1} {\bibfield  {journal} {\bibinfo
  {journal} {Physica D}\ }\textbf {\bibinfo {volume} {147}},\ \bibinfo {pages}
  {352--370} (\bibinfo {year} {2000})}\BibitemShut {NoStop}%
\bibitem [{\citenamefont {{P}eacock}\ and\ \citenamefont
  {{H}aller}(2013)}]{Peacock2013}%
  \BibitemOpen
  \bibfield  {author} {\bibinfo {author} {\bibfnamefont {T.}~\bibnamefont
  {{P}eacock}}\ and\ \bibinfo {author} {\bibfnamefont {G.}~\bibnamefont
  {{H}aller}},\ }\bibfield  {title} {\enquote {\bibinfo {title} {{L}agrangian
  coherent structures: {T}he hidden skeleton of fluid flows},}\ }\href
  {\doibase 10.1063/PT.3.1886} {\bibfield  {journal} {\bibinfo  {journal}
  {Physics Today}\ }\textbf {\bibinfo {volume} {66}},\ \bibinfo {pages}
  {41--47} (\bibinfo {year} {2013})}\BibitemShut {NoStop}%
\bibitem [{\citenamefont {{H}aller}(2001{\natexlab{a}})}]{Haller2001}%
  \BibitemOpen
  \bibfield  {author} {\bibinfo {author} {\bibfnamefont {G.}~\bibnamefont
  {{H}aller}},\ }\bibfield  {title} {\enquote {\bibinfo {title}
  {{D}istinguished material surfaces and coherent structures in
  three-dimensional fluid flows},}\ }\href {\doibase
  10.1016/S0167-2789(00)00199-8} {\bibfield  {journal} {\bibinfo  {journal}
  {Physica D}\ }\textbf {\bibinfo {volume} {149}},\ \bibinfo {pages} {248--277}
  (\bibinfo {year} {2001}{\natexlab{a}})}\BibitemShut {NoStop}%
\bibitem [{\citenamefont {{S}hadden}, \citenamefont {{L}ekien},\ and\
  \citenamefont {{M}arsden}(2005)}]{Shadden2005}%
  \BibitemOpen
  \bibfield  {author} {\bibinfo {author} {\bibfnamefont {S.~C.}\ \bibnamefont
  {{S}hadden}}, \bibinfo {author} {\bibfnamefont {F.}~\bibnamefont {{L}ekien}},
  \ and\ \bibinfo {author} {\bibfnamefont {J.~E.}\ \bibnamefont {{M}arsden}},\
  }\bibfield  {title} {\enquote {\bibinfo {title} {{D}efinition and properties
  of {L}agrangian coherent structures from finite-time {Lyapunov} exponents in
  two-dimensional aperiodic flows},}\ }\href {\doibase
  10.1016/j.physd.2005.10.007} {\bibfield  {journal} {\bibinfo  {journal}
  {Physica D}\ }\textbf {\bibinfo {volume} {212}},\ \bibinfo {pages} {271--304}
  (\bibinfo {year} {2005})}\BibitemShut {NoStop}%
\bibitem [{\citenamefont {{L}ekien}, \citenamefont {{S}hadden},\ and\
  \citenamefont {{M}arsden}(2007)}]{Lekien2007}%
  \BibitemOpen
  \bibfield  {author} {\bibinfo {author} {\bibfnamefont {F.}~\bibnamefont
  {{L}ekien}}, \bibinfo {author} {\bibfnamefont {S.~C.}\ \bibnamefont
  {{S}hadden}}, \ and\ \bibinfo {author} {\bibfnamefont {J.~E.}\ \bibnamefont
  {{M}arsden}},\ }\bibfield  {title} {\enquote {\bibinfo {title} {{L}agrangian
  coherent structures in n-dimensional systems},}\ }\href {\doibase
  10.1063/1.2740025} {\bibfield  {journal} {\bibinfo  {journal} {J. Math.
  Phys.}\ }\textbf {\bibinfo {volume} {48}},\ \bibinfo {pages} {065404.\ 1--19}
  (\bibinfo {year} {2007})}\BibitemShut {NoStop}%
\bibitem [{\citenamefont {{J}oseph}\ and\ \citenamefont
  {{L}egras}(2002)}]{Joseph2002}%
  \BibitemOpen
  \bibfield  {author} {\bibinfo {author} {\bibfnamefont {B.}~\bibnamefont
  {{J}oseph}}\ and\ \bibinfo {author} {\bibfnamefont {B.}~\bibnamefont
  {{L}egras}},\ }\bibfield  {title} {\enquote {\bibinfo {title} {{R}elation
  between {K}inematic {B}oundaries, {S}tirring, and {B}arriers for the
  {A}ntarctic {P}olar {V}ortex},}\ }\href {\doibase 10.1175/1520-0469(2002)059}
  {\bibfield  {journal} {\bibinfo  {journal} {J. Atmos. Sci.}\ }\textbf
  {\bibinfo {volume} {59}},\ \bibinfo {pages} {1198--1212} (\bibinfo {year}
  {2002})}\BibitemShut {NoStop}%
\bibitem [{\citenamefont {{B}owman}(1999)}]{Bowman1999}%
  \BibitemOpen
  \bibfield  {author} {\bibinfo {author} {\bibfnamefont {K.~P.}\ \bibnamefont
  {{B}owman}},\ }\href@noop {} {\enquote {\bibinfo {title} {{M}anifold
  {G}eometry and {M}ixing in {O}bserved {A}tmospheric {F}lows},}\ } (\bibinfo
  {year} {1999}),\ \bibinfo {note} {preprint}\BibitemShut {NoStop}%
\bibitem [{\citenamefont {{F}royland}\ and\ \citenamefont {{P}adberg
  {G}ehle}(2012)}]{Froyland2012}%
  \BibitemOpen
  \bibfield  {author} {\bibinfo {author} {\bibfnamefont {G.}~\bibnamefont
  {{F}royland}}\ and\ \bibinfo {author} {\bibfnamefont {K.}~\bibnamefont
  {{P}adberg {G}ehle}},\ }\bibfield  {title} {\enquote {\bibinfo {title}
  {{F}inite-time entropy: {A} probabilistic approach for measuring nonlinear
  stretching},}\ }\href {\doibase 10.1016/j.physd.2012.06.010} {\bibfield
  {journal} {\bibinfo  {journal} {Physica D}\ }\textbf {\bibinfo {volume}
  {241}},\ \bibinfo {pages} {1612--1628} (\bibinfo {year} {2012})}\BibitemShut
  {NoStop}%
\bibitem [{\citenamefont {{H}aller}(2001{\natexlab{b}})}]{Haller2001a}%
  \BibitemOpen
  \bibfield  {author} {\bibinfo {author} {\bibfnamefont {G.}~\bibnamefont
  {{H}aller}},\ }\bibfield  {title} {\enquote {\bibinfo {title} {{L}agrangian
  structures and the rate of strain in a partition of two-dimensional
  turbulence},}\ }\href {\doibase 10.1063/1.1403336} {\bibfield  {journal}
  {\bibinfo  {journal} {Physics of Fluids}\ }\textbf {\bibinfo {volume} {13}},\
  \bibinfo {pages} {3365--3385} (\bibinfo {year}
  {2001}{\natexlab{b}})}\BibitemShut {NoStop}%
\bibitem [{\citenamefont {{H}aller}(2002)}]{Haller2002}%
  \BibitemOpen
  \bibfield  {author} {\bibinfo {author} {\bibfnamefont {G.}~\bibnamefont
  {{H}aller}},\ }\bibfield  {title} {\enquote {\bibinfo {title} {{L}agrangian
  coherent structures from approximate velocity data},}\ }\href {\doibase
  10.1063/1.1477449} {\bibfield  {journal} {\bibinfo  {journal} {Physics of
  Fluids}\ }\textbf {\bibinfo {volume} {14}},\ \bibinfo {pages} {1851--1861}
  (\bibinfo {year} {2002})}\BibitemShut {NoStop}%
\bibitem [{\citenamefont {{B}ranicki}\ and\ \citenamefont
  {{W}iggins}(2010)}]{Branicki2010}%
  \BibitemOpen
  \bibfield  {author} {\bibinfo {author} {\bibfnamefont {M.}~\bibnamefont
  {{B}ranicki}}\ and\ \bibinfo {author} {\bibfnamefont {S.}~\bibnamefont
  {{W}iggins}},\ }\bibfield  {title} {\enquote {\bibinfo {title} {{F}inite-time
  {L}agrangian transport analysis: stable and unstable manifolds of hyperbolic
  trajectories and finite-time {Lyapunov} exponents},}\ }\href {\doibase
  10.5194/npg-17-1-2010} {\bibfield  {journal} {\bibinfo  {journal} {Nonlinear
  Processes in Geophysics}\ }\textbf {\bibinfo {volume} {17}},\ \bibinfo
  {pages} {1--36} (\bibinfo {year} {2010})}\BibitemShut {NoStop}%
\bibitem [{\citenamefont {{F}arazmand}\ and\ \citenamefont
  {{H}aller}(2012{\natexlab{a}})}]{Farazmand2012a}%
  \BibitemOpen
  \bibfield  {author} {\bibinfo {author} {\bibfnamefont {M.}~\bibnamefont
  {{F}arazmand}}\ and\ \bibinfo {author} {\bibfnamefont {G.}~\bibnamefont
  {{H}aller}},\ }\bibfield  {title} {\enquote {\bibinfo {title} {{C}omputing
  {L}agrangian coherent structures from their variational theory},}\ }\href
  {\doibase 10.1063/1.3690153} {\bibfield  {journal} {\bibinfo  {journal}
  {Chaos}\ }\textbf {\bibinfo {volume} {22}},\ \bibinfo {eid} {013128}
  (\bibinfo {year} {2012}{\natexlab{a}})}\BibitemShut {NoStop}%
\bibitem [{\citenamefont {{K}arrasch}(2012)}]{Karrasch2012}%
  \BibitemOpen
  \bibfield  {author} {\bibinfo {author} {\bibfnamefont {D.}~\bibnamefont
  {{K}arrasch}},\ }\bibfield  {title} {\enquote {\bibinfo {title} {{C}omment on
  ``{A} variational theory of hyperbolic {L}agrangian {C}oherent {S}tructures,
  {P}hysica {D} 240 (2011) 574--598''},}\ }\href {\doibase
  10.1016/j.physd.2012.05.008} {\bibfield  {journal} {\bibinfo  {journal}
  {Physica D}\ }\textbf {\bibinfo {volume} {241}},\ \bibinfo {pages}
  {1470--1473} (\bibinfo {year} {2012})}\BibitemShut {NoStop}%
\bibitem [{\citenamefont {{H}aller}\ and\ \citenamefont
  {{Beron-Vera}}(2012)}]{Haller2012}%
  \BibitemOpen
  \bibfield  {author} {\bibinfo {author} {\bibfnamefont {G.}~\bibnamefont
  {{H}aller}}\ and\ \bibinfo {author} {\bibfnamefont {F.~J.}\ \bibnamefont
  {{Beron-Vera}}},\ }\bibfield  {title} {\enquote {\bibinfo {title} {{G}eodesic
  theory of transport barriers in two-dimensional flows},}\ }\href {\doibase
  10.1016/j.physd.2012.06.012} {\bibfield  {journal} {\bibinfo  {journal}
  {Physica D}\ }\textbf {\bibinfo {volume} {241}},\ \bibinfo {pages}
  {1680--1702} (\bibinfo {year} {2012})}\BibitemShut {NoStop}%
\bibitem [{\citenamefont {{F}arazmand}, \citenamefont {{B}lazevski},\ and\
  \citenamefont {{H}aller}(2014)}]{Farazmand2014a}%
  \BibitemOpen
  \bibfield  {author} {\bibinfo {author} {\bibfnamefont {M.}~\bibnamefont
  {{F}arazmand}}, \bibinfo {author} {\bibfnamefont {D.}~\bibnamefont
  {{B}lazevski}}, \ and\ \bibinfo {author} {\bibfnamefont {G.}~\bibnamefont
  {{H}aller}},\ }\bibfield  {title} {\enquote {\bibinfo {title} {{S}hearless
  transport barriers in unsteady two-dimensional flows and maps},}\ }\href
  {\doibase 10.1016/j.physd.2014.03.008} {\bibfield  {journal} {\bibinfo
  {journal} {Physica D}\ }\textbf {\bibinfo {volume} {278-279}},\ \bibinfo
  {pages} {44--57} (\bibinfo {year} {2014})}\BibitemShut {NoStop}%
\bibitem [{\citenamefont {{K}arrasch}, \citenamefont {{F}arazmand},\ and\
  \citenamefont {{H}aller}()}]{Karrasch2013c}%
  \BibitemOpen
  \bibfield  {author} {\bibinfo {author} {\bibfnamefont {D.}~\bibnamefont
  {{K}arrasch}}, \bibinfo {author} {\bibfnamefont {M.}~\bibnamefont
  {{F}arazmand}}, \ and\ \bibinfo {author} {\bibfnamefont {G.}~\bibnamefont
  {{H}aller}},\ }\bibfield  {title} {\enquote {\bibinfo {title}
  {{A}ttraction-{B}ased {C}omputation of {H}yperbolic {L}agrangian {C}oherent
  {S}tructures},}\ }\href@noop {} {\bibfield  {journal} {\bibinfo  {journal}
  {Journal of Computational Dynamics}\ }}\bibinfo {note} {In press}\BibitemShut
  {NoStop}%
\bibitem [{\citenamefont {{L}ekien}\ and\ \citenamefont
  {{R}oss}(2010)}]{Lekien2010}%
  \BibitemOpen
  \bibfield  {author} {\bibinfo {author} {\bibfnamefont {F.}~\bibnamefont
  {{L}ekien}}\ and\ \bibinfo {author} {\bibfnamefont {S.~D.}\ \bibnamefont
  {{R}oss}},\ }\bibfield  {title} {\enquote {\bibinfo {title} {{T}he
  computation of finite-time {Lyapunov} exponents on unstructured meshes and
  for non-{Euclid}ean manifolds},}\ }\href {\doibase 10.1063/1.3278516}
  {\bibfield  {journal} {\bibinfo  {journal} {Chaos}\ }\textbf {\bibinfo
  {volume} {20}},\ \bibinfo {pages} {017505} (\bibinfo {year}
  {2010})}\BibitemShut {NoStop}%
\bibitem [{\citenamefont {{G}reene}\ and\ \citenamefont
  {{K}im}(1987)}]{Greene1987}%
  \BibitemOpen
  \bibfield  {author} {\bibinfo {author} {\bibfnamefont {J.~M.}\ \bibnamefont
  {{G}reene}}\ and\ \bibinfo {author} {\bibfnamefont {J.-S.}\ \bibnamefont
  {{K}im}},\ }\bibfield  {title} {\enquote {\bibinfo {title} {{T}he calculation
  of {L}yapunov spectra},}\ }\href {\doibase 10.1016/0167-2789(87)90076-5}
  {\bibfield  {journal} {\bibinfo  {journal} {Physica D}\ }\textbf {\bibinfo
  {volume} {24}},\ \bibinfo {pages} {213--225} (\bibinfo {year}
  {1987})}\BibitemShut {NoStop}%
\bibitem [{\citenamefont {{B}lazevski}\ and\ \citenamefont
  {{H}aller}(2014)}]{Blazevski2014}%
  \BibitemOpen
  \bibfield  {author} {\bibinfo {author} {\bibfnamefont {D.}~\bibnamefont
  {{B}lazevski}}\ and\ \bibinfo {author} {\bibfnamefont {G.}~\bibnamefont
  {{H}aller}},\ }\bibfield  {title} {\enquote {\bibinfo {title} {{H}yperbolic
  and elliptic transport barriers in three-dimensional unsteady flows},}\
  }\href {\doibase 10.1016/j.physd.2014.01.007} {\bibfield  {journal} {\bibinfo
   {journal} {Physica D}\ }\textbf {\bibinfo {volume} {273-274}},\ \bibinfo
  {pages} {46--62} (\bibinfo {year} {2014})}\BibitemShut {NoStop}%
\bibitem [{\citenamefont {{K}arrasch}\ and\ \citenamefont
  {{H}aller}(2013)}]{Karrasch2013a}%
  \BibitemOpen
  \bibfield  {author} {\bibinfo {author} {\bibfnamefont {D.}~\bibnamefont
  {{K}arrasch}}\ and\ \bibinfo {author} {\bibfnamefont {G.}~\bibnamefont
  {{H}aller}},\ }\bibfield  {title} {\enquote {\bibinfo {title} {{D}o
  {F}inite-{S}ize {Lyapunov} {E}xponents {D}etect {C}oherent {S}tructures?}}\
  }\href {\doibase 10.1063/1.4837075} {\bibfield  {journal} {\bibinfo
  {journal} {Chaos}\ }\textbf {\bibinfo {volume} {23}},\ \bibinfo {pages}
  {043126--} (\bibinfo {year} {2013})}\BibitemShut {NoStop}%
\bibitem [{\citenamefont {{E}berly}\ \emph {et~al.}(1994)\citenamefont
  {{E}berly}, \citenamefont {{G}ardner}, \citenamefont {{M}orse}, \citenamefont
  {{P}izer},\ and\ \citenamefont {{S}charlach}}]{Eberly1994}%
  \BibitemOpen
  \bibfield  {author} {\bibinfo {author} {\bibfnamefont {D.}~\bibnamefont
  {{E}berly}}, \bibinfo {author} {\bibfnamefont {R.}~\bibnamefont {{G}ardner}},
  \bibinfo {author} {\bibfnamefont {B.}~\bibnamefont {{M}orse}}, \bibinfo
  {author} {\bibfnamefont {S.}~\bibnamefont {{P}izer}}, \ and\ \bibinfo
  {author} {\bibfnamefont {C.}~\bibnamefont {{S}charlach}},\ }\bibfield
  {title} {\enquote {\bibinfo {title} {{R}idges for image analysis},}\ }\href
  {\doibase 10.1007/BF01262402} {\bibfield  {journal} {\bibinfo  {journal} {J.
  Math. Imaging Vis.}\ }\textbf {\bibinfo {volume} {4}},\ \bibinfo {pages}
  {353--373} (\bibinfo {year} {1994})}\BibitemShut {NoStop}%
\bibitem [{\citenamefont {{T}ruesdell}\ and\ \citenamefont
  {{N}oll}(2004)}]{Truesdell2004}%
  \BibitemOpen
  \bibfield  {author} {\bibinfo {author} {\bibfnamefont {C.}~\bibnamefont
  {{T}ruesdell}}\ and\ \bibinfo {author} {\bibfnamefont {W.}~\bibnamefont
  {{N}oll}},\ }\href {\doibase 10.1007/978-3-662-10388-3} {\emph {\bibinfo
  {title} {{T}he {N}on-{L}inear {F}ield {T}heories of {M}echanics}}},\ edited
  by\ \bibinfo {editor} {\bibfnamefont {S.~S.}\ \bibnamefont {Antman}}\
  (\bibinfo  {publisher} {Springer},\ \bibinfo {year} {2004})\BibitemShut
  {NoStop}%
\bibitem [{\citenamefont {{A}braham}, \citenamefont {{M}arsden},\ and\
  \citenamefont {{R}atiu}(1988)}]{Abraham1988}%
  \BibitemOpen
  \bibfield  {author} {\bibinfo {author} {\bibfnamefont {R.}~\bibnamefont
  {{A}braham}}, \bibinfo {author} {\bibfnamefont {J.~E.}\ \bibnamefont
  {{M}arsden}}, \ and\ \bibinfo {author} {\bibfnamefont {T.}~\bibnamefont
  {{R}atiu}},\ }\href {\doibase 10.1007/978-1-4612-1029-0} {\emph {\bibinfo
  {title} {{M}anifolds, {T}ensor {A}nalysis, and {A}pplications}}},\ \bibinfo
  {edition} {2nd}\ ed.,\ \bibinfo {series} {Applied Mathematical Sciences},
  Vol.~\bibinfo {volume} {75}\ (\bibinfo  {publisher} {Springer},\ \bibinfo
  {year} {1988})\BibitemShut {NoStop}%
\bibitem [{\citenamefont {{M}ezic}\ \emph {et~al.}(2010)\citenamefont
  {{M}ezic}, \citenamefont {{L}oire}, \citenamefont {{F}onoberov},\ and\
  \citenamefont {{H}ogan}}]{Mezic2010}%
  \BibitemOpen
  \bibfield  {author} {\bibinfo {author} {\bibfnamefont {I.}~\bibnamefont
  {{M}ezic}}, \bibinfo {author} {\bibfnamefont {S.}~\bibnamefont {{L}oire}},
  \bibinfo {author} {\bibfnamefont {V.~A.}\ \bibnamefont {{F}onoberov}}, \ and\
  \bibinfo {author} {\bibfnamefont {P.}~\bibnamefont {{H}ogan}},\ }\bibfield
  {title} {\enquote {\bibinfo {title} {{A} {N}ew {M}ixing {D}iagnostic and
  {G}ulf {O}il {S}pill {M}ovement},}\ }\href {\doibase 10.1126/science.1194607}
  {\bibfield  {journal} {\bibinfo  {journal} {Science}\ }\textbf {\bibinfo
  {volume} {330}},\ \bibinfo {pages} {486--489} (\bibinfo {year}
  {2010})}\BibitemShut {NoStop}%
\bibitem [{\citenamefont {{T}refethen}\ and\ \citenamefont
  {{B}au}(1997)}]{Trefethen1997}%
  \BibitemOpen
  \bibfield  {author} {\bibinfo {author} {\bibfnamefont {L.~N.}\ \bibnamefont
  {{T}refethen}}\ and\ \bibinfo {author} {\bibfnamefont {D.}~\bibnamefont
  {{B}au}},\ }\href@noop {} {\emph {\bibinfo {title} {{N}umerical linear
  algebra}}}\ (\bibinfo  {publisher} {SIAM},\ \bibinfo {year}
  {1997})\BibitemShut {NoStop}%
\bibitem [{\citenamefont {{L}ee}(2012)}]{Lee2012}%
  \BibitemOpen
  \bibfield  {author} {\bibinfo {author} {\bibfnamefont {J.~M.}\ \bibnamefont
  {{L}ee}},\ }\href {\doibase 10.1007/978-1-4419-9982-5} {\emph {\bibinfo
  {title} {{I}ntroduction to {S}mooth {M}anifolds}}},\ \bibinfo {edition}
  {2nd}\ ed.,\ \bibinfo {series} {Graduate Texts in Mathematics}, Vol.\
  \bibinfo {volume} {218}\ (\bibinfo  {publisher} {Springer},\ \bibinfo {year}
  {2012})\BibitemShut {NoStop}%
\bibitem [{\citenamefont {{P}obitzer}\ \emph {et~al.}(2012)\citenamefont
  {{P}obitzer}, \citenamefont {{P}eikert}, \citenamefont {{F}uchs},
  \citenamefont {{T}heisel},\ and\ \citenamefont {{H}auser}}]{Pobitzer2012}%
  \BibitemOpen
  \bibfield  {author} {\bibinfo {author} {\bibfnamefont {A.}~\bibnamefont
  {{P}obitzer}}, \bibinfo {author} {\bibfnamefont {R.}~\bibnamefont
  {{P}eikert}}, \bibinfo {author} {\bibfnamefont {R.}~\bibnamefont {{F}uchs}},
  \bibinfo {author} {\bibfnamefont {H.}~\bibnamefont {{T}heisel}}, \ and\
  \bibinfo {author} {\bibfnamefont {H.}~\bibnamefont {{H}auser}},\ }\bibfield
  {title} {\enquote {\bibinfo {title} {{F}iltering of {FTLE} for {V}isualizing
  {S}patial {S}eparation in {U}nsteady {3D} {F}low},}\ }in\ \href {\doibase
  10.1007/978-3-642-23175-9_16} {\emph {\bibinfo {booktitle} {Topological
  Methods in Data Analysis and Visualization II}}},\ \bibinfo {series and
  number} {Mathematics and Visualization},\ \bibinfo {editor} {edited by\
  \bibinfo {editor} {\bibfnamefont {R.}~\bibnamefont {Peikert}}, \bibinfo
  {editor} {\bibfnamefont {H.}~\bibnamefont {Hauser}}, \bibinfo {editor}
  {\bibfnamefont {H.}~\bibnamefont {Carr}}, \ and\ \bibinfo {editor}
  {\bibfnamefont {R.}~\bibnamefont {Fuchs}}}\ (\bibinfo  {publisher}
  {Springer},\ \bibinfo {year} {2012})\ pp.\ \bibinfo {pages}
  {237--253}\BibitemShut {NoStop}%
\bibitem [{\citenamefont {{M}a}\ and\ \citenamefont {{B}ollt}(2014)}]{Ma2014}%
  \BibitemOpen
  \bibfield  {author} {\bibinfo {author} {\bibfnamefont {T.}~\bibnamefont
  {{M}a}}\ and\ \bibinfo {author} {\bibfnamefont {E.~M.}\ \bibnamefont
  {{B}ollt}},\ }\bibfield  {title} {\enquote {\bibinfo {title} {{D}ifferential
  {G}eometry {P}erspective of {S}hape {C}oherence and {C}urvature {E}volution
  by {F}inite-{T}ime {N}on-hyperbolic {S}plitting},}\ }\href {\doibase
  10.1137/130940633} {\bibfield  {journal} {\bibinfo  {journal} {SIAM J. Appl.
  Dyn. Syst.}\ }\textbf {\bibinfo {volume} {13}},\ \bibinfo {pages}
  {1106--1136} (\bibinfo {year} {2014})}\BibitemShut {NoStop}%
\bibitem [{Note1()}]{Note1}%
  \BibitemOpen
  \bibinfo {note} {See also the discussion in Ref.\ \protect \rev@citealpnum
  {Farazmand2013}}\BibitemShut {NoStop}%
\bibitem [{\citenamefont {{F}arazmand}\ and\ \citenamefont
  {{H}aller}(2012{\natexlab{b}})}]{Farazmand2012}%
  \BibitemOpen
  \bibfield  {author} {\bibinfo {author} {\bibfnamefont {M.}~\bibnamefont
  {{F}arazmand}}\ and\ \bibinfo {author} {\bibfnamefont {G.}~\bibnamefont
  {{H}aller}},\ }\bibfield  {title} {\enquote {\bibinfo {title} {{E}rratum and
  addendum to ``{A} variational theory of hyperbolic {L}agrangian coherent
  structures [{P}hysica {D} 240 (2011) 574--598]''},}\ }\href {\doibase
  10.1016/j.physd.2011.09.013} {\bibfield  {journal} {\bibinfo  {journal}
  {Physica D}\ }\textbf {\bibinfo {volume} {241}},\ \bibinfo {pages} {439--441}
  (\bibinfo {year} {2012}{\natexlab{b}})}\BibitemShut {NoStop}%
\bibitem [{\citenamefont {{E}berly}(1996)}]{Eberly1996}%
  \BibitemOpen
  \bibfield  {author} {\bibinfo {author} {\bibfnamefont {D.}~\bibnamefont
  {{E}berly}},\ }\href@noop {} {\emph {\bibinfo {title} {{R}idges in {I}mage
  and {D}ata {A}nalysis}}},\ \bibinfo {series} {Computational Imaging and
  Vision}, Vol.~\bibinfo {volume} {7}\ (\bibinfo  {publisher} {Kluwer Academic
  Publishers},\ \bibinfo {year} {1996})\BibitemShut {NoStop}%
\bibitem [{\citenamefont {{P}almerius}, \citenamefont {{C}ooper},\ and\
  \citenamefont {{Y}nnerman}(2009)}]{Palmerius2009}%
  \BibitemOpen
  \bibfield  {author} {\bibinfo {author} {\bibfnamefont {K.~L.}\ \bibnamefont
  {{P}almerius}}, \bibinfo {author} {\bibfnamefont {M.}~\bibnamefont
  {{C}ooper}}, \ and\ \bibinfo {author} {\bibfnamefont {A.}~\bibnamefont
  {{Y}nnerman}},\ }\bibfield  {title} {\enquote {\bibinfo {title} {{F}low field
  visualization using vector field perpendicular surfaces},}\ }in\ \href
  {\doibase 10.1145/1980462.1980471} {\emph {\bibinfo {booktitle} {Proceedings
  of the 25\textsuperscript{th} Spring Conference on Computer Graphics}}},\
  \bibinfo {series and number} {SCCG '09}\ (\bibinfo  {publisher} {ACM},\
  \bibinfo {year} {2009})\ pp.\ \bibinfo {pages} {27--34}\BibitemShut {NoStop}%
\bibitem [{\citenamefont {{N}orgard}\ and\ \citenamefont
  {{B}remer}(2012)}]{Norgard2012}%
  \BibitemOpen
  \bibfield  {author} {\bibinfo {author} {\bibfnamefont {G.}~\bibnamefont
  {{N}orgard}}\ and\ \bibinfo {author} {\bibfnamefont {P.-T.}\ \bibnamefont
  {{B}remer}},\ }\bibfield  {title} {\enquote {\bibinfo {title} {{S}econd
  derivative ridges are straight lines and the implications for computing
  {L}agrangian {C}oherent {S}tructures},}\ }\href {\doibase
  10.1016/j.physd.2012.05.006} {\bibfield  {journal} {\bibinfo  {journal}
  {Physica D}\ }\textbf {\bibinfo {volume} {241}},\ \bibinfo {pages} {1475 --
  1476} (\bibinfo {year} {2012})}\BibitemShut {NoStop}%
\bibitem [{\citenamefont {{A}llaire}\ and\ \citenamefont
  {{K}aber}(2008)}]{Allaire2008}%
  \BibitemOpen
  \bibfield  {author} {\bibinfo {author} {\bibfnamefont {G.}~\bibnamefont
  {{A}llaire}}\ and\ \bibinfo {author} {\bibfnamefont {S.~M.}\ \bibnamefont
  {{K}aber}},\ }\href {\doibase 10.1007/978-0-387-68918-0} {\emph {\bibinfo
  {title} {{N}umerical {L}inear {A}lgebra}}},\ \bibinfo {series} {Texts in
  Applied Mathematics}, Vol.~\bibinfo {volume} {55}\ (\bibinfo  {publisher}
  {Springer},\ \bibinfo {year} {2008})\BibitemShut {NoStop}%
\end{thebibliography}

%

\end{document}